\documentclass[12pt,a4paper,reqno]{amsart}

\usepackage{amsthm}
\usepackage{amsfonts}
\usepackage{amsmath}
\usepackage{mathrsfs}
\usepackage{amssymb}
\usepackage{amsfonts}
\usepackage{amsbsy}

\usepackage{CJK,CJKnumb}
\usepackage{color}
\usepackage{indentfirst}
\usepackage{latexsym,bm}
\usepackage{graphicx}
\usepackage{cases}
\usepackage{pifont}
\usepackage{txfonts}
\usepackage{xcolor}
\usepackage[all,knot,poly]{xy}

\usepackage[labelformat=empty]{caption}
\usepackage{enumitem}

\usepackage{ulem}

\newtheorem{theo}{Theorem}[section]
\newtheorem{lem}[theo]{Lemma}
\newtheorem{coro}[theo]{Corollary}
\newtheorem{prop}[theo]{Proposition}

\theoremstyle{definition}
\newtheorem{defi}[theo]{Definition}

\theoremstyle{remark}
\newtheorem{rem}[theo]{Remark}
\newtheorem{example}[theo]{Example}

\numberwithin{equation}{section}

\newcommand{\C}{{\mathbb C}}

\newcommand{\Z}{{\mathbb Z}}

\newcommand{\osp}{{\rm\mathfrak{osp}}}
\newcommand{\Sl}{{\rm\mathfrak{sl}}}

\newcommand{\g}{{\mathfrak g}}

\newcommand{\de}{\delta}
\newcommand{\ga}{\gamma}

\newcommand{\Q}{{\mathcal Q}}
\newcommand{\hQ}{{\widehat{\mathcal Q}}}

\newcommand{\cq}{{\mathcal{C}_q}}
\newcommand{\cqq}{{\mathfrak{C}_q}}

\newcommand{\Uq}{{{\rm  U}_q}}
\newcommand{\U}{{\rm  U }}

\newcommand{\cI}{{\mathcal I}}

\newcommand{\e}{\xi^+ }
\newcommand{\f}{\xi^-}
\newcommand{\ef}{\xi^{\pm}}
\newcommand{\qe}{e}
\newcommand{\qf}{f}
\newcommand{\x}{\xi'}
\newcommand{\h}{\kappa}
\newcommand{\qk}{\gamma}
\newcommand{\hh}{\hat{\kappa}}

\newcommand{\X}{X}

\newcommand{\HH}{H}

\newcommand{\ka}{\mathfrak{k}}
\newcommand{\ta}{\mathfrak{t}}

\newcommand{\xp}{{\rm {exp}}}

\newcommand{\up}{\Upsilon}

\newcommand{\fU}{{\mathfrak U}}

\allowdisplaybreaks

\begin{document}

\title[Vertex operator representations]{Vertex operator representations of  \\ quantum affine superalgebras}

\author{Ying Xu}
\author{R. B. Zhang}
\address[Xu]{School of Mathematics, Hefei University of Technology, Hefei, China}
\address[Xu, Zhang]{School of Mathematics and Statistics,
University of Sydney, NSW 2006, Australia}
\email{xuying@hfut.edu.cn, ruibin.zhang@sydney.edu.au}

\begin{abstract}
Let $\Uq(\g)$ be the quantum affine superalgebra associated with an affine Kac-Moody superalgebra $\g$ which belongs to the three series $\osp(1|2n)^{(1)}$, $\Sl(1|2n)^{(2)}$ and $\osp(2|2n)^{(2)}$. We develop vertex operator constructions for the level $1$ irreducible integrable highest weight representations and classify the finite dimensional irreducible representations of $\Uq(\g)$.  This makes essential use of the Drinfeld realisation for $\Uq(\g)$, and quantum correspondences between affine Kac-Moody superalgebras, developed in earlier papers.
\end{abstract}
\subjclass[2010]{17B37,17B69}
\keywords{affine superalgebras;quantum affine superalgebras;Drinfeld realisation;Vertex operator representations;finite dimensional irreducible representations}

\maketitle


\section{Introduction}\label{intro}
Quantum supergroups and quantum affine superalgebras were introduced in the early 90s  \cite{BGZ, Y94, Y99}, which have important applications in a number of areas such as low dimensional topology \cite{LGZ, Z92a, Z95} and statistical mechanics \cite{BGZ, ZBG91}.
There was extensive work in the 90s on the representation theory of quantum supergroups associated with finite dimensional simple Lie superalgebras \cite{BKK, PSV, Zhc, Z92b, Z93, Z93b, Zo}, and the $\mathfrak{gl}(m|n)$ super Yangian \cite{Z96} (see \cite{Z98} for a review of early results).  Representations of some classes of quantum affine superalgebras were also studied: the integrable highest weight representations were well understood \cite{Z2} for the quantum affine superalgebras associated with the affine Lie superalgebras without isotropic real roots; the finite dimensional irreducible representations and evaluation representations of the untwisted quantum affine superalgebra $\Uq(\mathfrak{gl}(m|n)^{(1)})$ were thoroughly treated in recent years \cite{WZ,Zh}. We mention in particular  that finite dimensional representations of the quantum supergroup $\Uq(\mathfrak{gl}(m|n))$ all lift to evaluation representations of $\Uq(\mathfrak{gl}(m|n)^{(1)})$,
a fact which has long been known \cite{Z92c}.

One of the problems hindering progress in the study of quantum affine superalgebras was the lack of Drinfeld realisations \cite{De, Dr} except for the case of  $\Uq(\mathfrak{gl}(m|n)^{(1)})$ \cite{Y99}. This is rectified \cite{XZ1} recently for the quantum affine superalgebras
$\Uq(\g)$ associated with the
affine Kac-Moody superalgebras $\g$ belonging to the three series given in \eqref{eq:g}.
The Drinfeld realisations of the quantum affine superalgebras will be used in an essential way in this paper.

Let $\g$ be an affine Kac-Moody superalgebra given in \eqref{eq:g}.
We will construct vertex operator representations
and classify the finite dimensional irreducible representations of $\Uq(\g)$.
The main results  are given in Theorem \ref{them:v.o} (and its variations in Sections \ref{sect:other-level-1}  and \ref{sect:other-vacuum}) and Theorem \ref{theo-finite module}.
Finite dimensional representations of quantum affine superalgebras play a crucial role in constructing soluble models of Yang-Baxter type \cite{BGZ}; vertex  operator representations have a direct connection with conformal field theory. Results of this paper are potentially applicable to mathematical physics.

The vertex operator representations of $\Uq(\g)$ constructed here are realised on quantum Fock spaces; they are level $1$ irreducible integrable highest weight representations relative to the standard triangular decomposition. Recall that there exists a well defined notion of
integrable highest weight representations \cite{Z2} for $\Uq(\g)$ with $\g$ belonging to \eqref{eq:g},  even though this is not true for most of the other quantum affine superalgebras.
We point out that our construction here is heavily influenced by work of Jing \cite{Jn1, JnM} on  the vertex operator representations of ordinary quantum affine algebras.

The finite dimensional irreducible representations of $\Uq(\g)$ are shown to be level $0$ highest weight representations relative to another triangular decomposition.
We obtain the necessary and sufficient conditions on the highest weights for
irreducible highest weight representations to be finite dimensional. The conditions are described in terms of Drinfeld’s highest weight polynomials.  The proof of the classification theorem (Theorem \ref{theo-finite module}) makes essential use of results of Chari and Pressley in \cite{CP0, CP1, CP2} on ordinary quantum affine algebras. Another important ingredient in the proof is quantum correspondences between affine Lie superalgebras developed in \cite{XZ, Z2}.  Some of the quantum correspondences appear as S-dualities in string theory in
work of Mikhaylov and Witten \cite{MW}.

\section{Drinfeld realisation of quantum affine superalgebras}\label{quantum}
Consider the following affine Kac-Moody superalgebras
\begin{eqnarray}\label{eq:g}
\osp(1|2n)^{(1)},  \quad \Sl(1|2n)^{(2)}, \quad \osp(2|2n)^{(2)}, \quad n\ge 1.
\end{eqnarray}
Here the notation is as in \cite{K78}, with $\osp(1|2n)^{(1)}$ denoting the untwisted affine  Lie superalgebra of $\osp(1|2n)$,  and $\osp(2|2n)^{(2)}$ and $\Sl(1|2n)^{(2)}$ the
twisted (by order two automorphisms) affine Lie superalgebras of $\osp(2|2n)$ and $\Sl(1|2n)$ respectively.
The Dynkin diagrams of the affine Lie superalgebras are as follows.
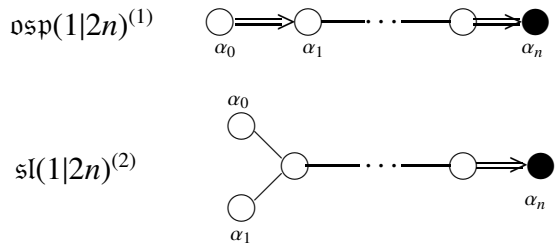
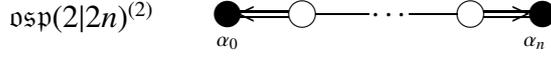
\[
\begin{picture}(120, 28)(40,-12)
\put(-42,-5) {$\osp(1|2n)^{(1)}$}
\put(37,0){\circle{10}}
\put(35,-12){\tiny\mbox{$\alpha_0$}}
\put(43,1){\line(1, 0){17}}
\put(43,-1){\line(1, 0){17}}
\put(58,-3){$>$}
\put(70, 0){\circle{10}}
\put(68,-12){\tiny\mbox{$\alpha_1$}}
\put(75, 0){\line(1, 0){16}}
\put(91, -0.5){\dots}
\put(105, 0){\line(1, 0){18}}
\put(128, 0){\circle{10}}
\put(132,1){\line(1, 0){17}}
\put(132,-1){\line(1, 0){17}}
\put(143,-3){$>$}
\put(155, 0){\circle*{10}}
\put(150,-12){\mbox{\tiny$\alpha_{n}$}}
\end{picture}
\]
\[
\begin{picture}(150, 60)(-20,-28)
\put(-85,-5) {$\Sl(1|2n)^{(2)}$}
\put(0, 15){\circle{10}}
\put(-5,23){\tiny$\alpha_0$}
\put(0, -16){\circle{10}}
\put(-4,-28){\tiny$\alpha_1$}
\put(15, -3){\line(-1, -1){10}}
\put(15, 3){\line(-1, 1){10}}
\put(20, 0){\circle{10}}
\put(24, 0){\line(1, 0){20}}
\put(46, -0.5){\dots}
\put(60,0){\line(1, 0){18}}
\put(83, 0){\circle{10}}
\put(88,1){\line(1, 0){17}}
\put(88,-1){\line(1, 0){17}}
\put(100,-3){$>$}
\put(112, 0){\circle*{10}}
\put(106,-15){\tiny$\alpha_{n}$}
\end{picture}
\]
\[
\begin{picture}(150, 30)(-10,-14)
\put(-75,-5){$\osp(2|2n)^{(2)}$}
\put(7,0){\circle*{10}}
\put(3,-12){\tiny $\alpha_0$}
\put(12,1){\line(1, 0){18}}
\put(12,-1){\line(1, 0){18}}
\put(10,-3){$<$}
\put(35, 0){\circle{10}}
\put(40, 0){\line(1, 0){20}}
\put(61, -0.5){\dots}
\put(75, 0){\line(1, 0){18}}
\put(98, 0){\circle{10}}
\put(103,1){\line(1, 0){17}}
\put(103,-1){\line(1, 0){17}}
\put(113,-3){$>$}
\put(125, 0){\circle*{10}}
\put(115,-12){ \tiny$\alpha_{n}$}
\end{picture}
\]
More details on their root systems can be found in \cite{K78} (also see \cite{XZ, Z2}).
Note in particular that these affine Lie superalgebras do not have isotropic odd roots.

Let $\g$ be an affine Lie superalgebra in \eqref{eq:g}.
We denote by
$A=(a_{ij})$ its Cartan matrix, which is realised in terms of the set of simple roots
$\Pi=\{\alpha_i \mid i=0, 1, 2, \dots, n\}$
with $a_{i j} = \frac{2(\alpha_i, \alpha_j)}{(\alpha_i, \alpha_i)}$. A simple root
$\alpha_i$ is odd if the corresponding node in the Dynkin diagram is black, and is even otherwise.
%
%

Let $q^{1/2}$ be an indeterminate, and let $\C(q^{1/2})$
be the field of rational functions in $q^{1/2}$. Denote
$[k]_z=\frac{z^k-z^{-k}}{z-z^{-1}}$, \
$[N]_z!=\prod_{i=1}^N[i]_z$ with $[0]_z!=1$,
and
$\begin{bmatrix} N\\k\end{bmatrix}_z=\frac{[N]_z!}{[N-k]_z![k]_z!}$.
Set $q_i=q^{\frac{(\alpha_i,\alpha_i)}{2}}$ for all $\alpha_i\in\Pi$.

\begin{defi}[\cite{Z2}]\label{defi:quantum-super}
Assume that $\g$ is one of the affine Lie superalgebras $\osp(1|2n)^{(1)}$, $\Sl(1|2n)^{(2)}$ and $\osp(2|2n)^{(2)}$.
The quantum affine  superalgebra $\Uq(\g)$  is an associative superalgebra over $\C(q^{1/2})$ with identity generated by the homogeneous
elements $\qe_i,\qf_i, k_i^{\pm1/2}$ ($0\le i \le n$), where $\qe_j,\qf_j$ are odd for odd simple roots $\alpha_j$, and the other generators are even, with the following defining relations:
\begin{eqnarray}
\nonumber
&& k_i^{\pm1/2}  k_i^{\mp1/2}= k_i^{\mp1/2}  k_i^{\pm1/2}=1,\quad  k_i^{1/2}  k_j^{1/2}= k_j^{1/2}  k_i^{1/2},\\
\nonumber
&&     k_i^{\pm1/2} \qe_j  k_i^{\mp1/2} = q_i^{a_{ij}/2}  \qe_j,
\quad  k_i^{\pm1/2} \qf_j   k_i^{\mp1/2} = q_i^{-a_{ij}/2} \qf_j,\\
\label{eq:xx-q}
&&\qe_i\qf_j     -    (-1)^{ [\qe_i][\qf_j] } \qf_j\qe_i
                   =\de_{ij}  \dfrac{  k_i- k_i^{-1} }
                                          { q_i-q_i^{-1} }, \quad \forall i, j, \\
\nonumber
&&\left(
            \mbox{Ad}_{\qe_i}    \right)^{1-a_{ij}}    (\qe_j)
    =\left(
             \mbox{Ad}_{\qf_i}    \right)^{1-a_{ij}}    (\qf_j)=0, \quad       \text{ if } i\neq j,
\end{eqnarray}
where $k_i^{\pm}=\left(k_i^{\pm1/2}\right)^2$, $\mbox{Ad}_{\qe_i}(x)$ and $\mbox{Ad}_{\qf_i}(x)$ are respectively defined by
\begin{equation}\label{eq:ad}
\begin{aligned}
\mbox{Ad}_{\qe_i}(x)     =         \qe_ix   -(-1)^{[\qe_i][x]}   k_i x  k_i^{-1} \qe_i,\\
\mbox{Ad}_{\qf_i}(x)      =         \qf_ix    -(-1)^{[\qf_i][x]}    k_i^{-1} x  k_i \qf_i.
\end{aligned}
\end{equation}
\end{defi}

For any $x, y\in \Uq(\g)$ and $a\in\C(q^{1/2})$, we shall write
\[
[x, y]_a=x y - (-1)^{[x][y]} a y x, \quad [x, y] = [x, y]_1.
\]
Then $\mbox{Ad}_{\qe_i}(\qe_j)=[\qe_i, \qe_j]_{q_i^{a_{ij}}}$ and
$\mbox{Ad}_{\qf_i}(\qf_j)=[\qf_i, \qf_j]_{q_i^{a_{ij}}}$.

\subsection{Drinfeld realisations}
In a recent paper \cite{XZ1},  we constructed the Drinfeld realisations
of the quantum affine superalgebras $\U_q(\g)$.
To describe the Drinfeld realisations,
we let $\cI=\{ (i,r)\mid 1\le i\le n,  \ r\in\Z \}$.  Define the set $\cI_\g$ by
$\cI_{\g}:=\cI$
if $\g=\osp(1|2n)^{(1)}$ or $\Sl(1|2n)^{(2)}$; and
$\cI_{\g}:=\cI\backslash \{ (i,2r+1)\mid 1\le i<n, \  r\in \Z\}$
if  $\g=\osp(2|2n)^{(2)}$.
Let $ \mathcal{I}_{\g}^* =\{(i, s)\in \mathcal{I}_{\g}\mid s\ne 0\}$.
Also, for any expression $f(x_{r_1},\dots,x_{r_k})$ in $x_{r_1},\dots,x_{r_k}$, we use $sym_{r_1,\dots,r_k}f(x_{r_1},\dots,x_{r_k})$ to denote $\sum_{\sigma}f(x_{\sigma(r_1)},\dots,x_{\sigma(r_k)})$, where the sum is over the permutation group of the set  $\{r_1, r_2, \dots, r_k\}$.

\begin{prop}[Drinfeld realisation \cite{XZ1}]  \label{prop:dr}
For $\g=\osp(1|2n)^{(1)}$, $\Sl(1|2n)^{(2)}$ or $\osp(2|2n)^{(2)}$, the quantum affine superalgebra $\U_q(\g)$ is generated by
\begin{align}\label{eq:generator}
\ef_{i,r}, \  \h_{i,s}, \ \qk_i^{\pm1/2}, \ \ga^{\pm 1/2}, \quad \text{for }\  (i,r)\in\mathcal{I}_{\g},   \   (i,s)\in\mathcal{I}_{\g}^*, \ 1\le i \le n,
\end{align}
where $\e_{n,r},\f_{n,r}$ are odd and the other generators are even, with the following defining relations
\begin{itemize}
\item[\rm(1)]  \hspace{8mm} $\ga^{\pm 1/2}$ are central, and $\ga^{1/2} \ga^{- 1/2}=1$,
\begin{align}
&\qk_i^{\pm1/2}\qk_i^{\mp1/2}=\qk_i^{\mp1/2}\qk_i^{\pm1/2}=1,\quad  \qk_i^{1/2}\qk_j^{1/2}=\qk_j^{1/2}\qk_i^{1/2}, \nonumber\\
 \label{eq:hh}
&[\h_{i,r},\h_{j,s}]=\delta_{r+s,0} \dfrac{ u_{i,j,r} (\ga^{r}-\ga^{-r})  }
                                                           { r   (q_i-q_i^{-1})(q_j-q_j^{-1})  }  ,\\
\label{eq:hx}
& \qk_i^{\pm1/2} \ef_{j,r} \qk_i^{\mp1/2}=q_i^{\pm a_{ij}/2} \ef_{j,r},\quad
 [\h_{i,r},\ef_{j,s}]  = \dfrac{  u_{i,j,r} \ga^{\mp|r|/2}  }
                                                   {   r(q_i-q_i^{-1})  }
                                                    \ef_{j,s+r},  \\
 \label{eq:xx}
& [\e_{i,r}, \f_{j,s}]  =\delta_{i,j}
                                \dfrac{   \ga^{\frac{r-s}{2}} \hh^{+}_{i,r+s}
                                          -  \ga^{\frac{s-r}{2}} \hh^{-}_{i,r+s}   }
                                         {  q_i-q_i^{-1}  },
\end{align}
where the $u_{i,j,r}$ are given in \eqref{eq:u-def};  and $\hh^{\pm}_{i,\pm r}$ are defined by
\begin{equation}\label{eq:hh-hat}
\begin{aligned}
&\sum_{r\in\Z}  \hh^{+}_{i,r}u^{-r} =\qk_i \xp  \left(
                                  (q_i-q_i^{-1})\sum_{r>0}\h_{i, r}u^{-r}
                                                               \right),\\
&\sum_{r\in\Z}  \hh^{-}_{i,-r}u^r    =\qk_i^{-1}  \xp  \left(
                                 (q_i^{-1}-q_i)\sum_{r>0}\h_{i,-r}u^r
                                                                \right),
\end{aligned}
\end{equation}
in which $\qk_i^{\pm1}=\left(\qk_i^{\pm1/2}\right)^2$;

\item[\rm(2)] {\rm Serre relations}

\begin{itemize}
\item[\rm (A)] $(i,j)\neq (n, n)$, and if $\g\neq \osp(1|2n)^{(1)}$,
 \begin{align}
\label{eq:xrs-xsr}
&[\ef_{i,r\pm \theta}, \ef_{j,s}]_{q^{a_{ij}}_{i}}+[ \ef_{j,s\pm \theta},\ef_{i,r}]_{q^{a_{ji}}_{j}}
     =0,
\end{align}
where $\theta=2$ if $\g=\osp(2|2n)^{(2)}$ and $1$ if $\g=\Sl(1|2n)^{(2)}$;

\item[\rm (B)] $n\ne i\neq j$, \ $\ell=1-a_{i j}$,
\[\begin{aligned}
\hspace{20mm}
sym_{r_1,\dots,r_\ell}\sum_{k=0}^\ell  (-1)^k
                                             \begin{bmatrix} \ell\\k\end{bmatrix}_{q_i}
                                              \ef_{i,r_1}\dots\ef_{i,r_k} \ef_{j,s}\ef_{i,r_k+1}\dots\ef_{i,r_\ell}=0;
\end{aligned}\]

\item[\rm (C)]  $n=i\ne j$, \ $\ell=1-a_{i j}$, and if $\g\ne \Sl(1|2n)^{(2)}$, $j<n-1$,
\[\begin{aligned}
\hspace{20mm}
sym_{r_1,\dots,r_\ell}\sum_{k=0}^\ell
                                              \begin{bmatrix} \ell\\k\end{bmatrix}_{\tilde{q_i}}
                                              \ef_{i,r_1}\dots\ef_{i,r_k} \ef_{j,s}\ef_{i,r_k+1}\dots\ef_{i,r_\ell}=0,
\end{aligned}\]
where $\tilde{q_i}=(-1)^{1/2}q_i$;

\item[\rm (D)] for $\g=\osp(1|2n)^{(1)}$,
 \begin{align*}
&sym_{r_1,r_2,r_3} [  [\ef_{n,r_1\pm 1},\ef_{n,r_2}]_{q_n^2},  \ef_{n,r_3}]_{q_n^4}=0;\\
&sym_{r,s}\Big([\ef_{n,r\pm 2}, \ef_{n,s}]_{q_n^2} -q_n^4 [\ef_{n,r\pm 1},  \ef_{n,s\pm 1} ]_{q_n^{-6}}\Big)=0;\\
&sym_{r,s}\Big(q_n^2[[\ef_{n,r\pm1},\ef_{n,s}]_{q_n^2},\ef_{n-1,k}]_{q_n^4}\\
&+(q_n^2+q_n^{-2})[[\ef_{n-1,k},\ef_{n,r\pm1}]_{q_n^2},\ef_{n,s}]\Big)=0;
\end{align*}

\item[\rm (E) ] for $\g=\osp(2|2n)^{(2)}$,
\[
sym_{r, s} [  [\ef_{n-1 ,k},\ef_{n, r\pm 1}]_{q_n^2},  \ef_{n, s}]=0.
\]
\end{itemize}
\end{itemize}
In the above, the scalars $u_{i,j,r}$ ($r\in\Z$, $i, j=1, 2, \dots, n$)  are defined by
\begin{eqnarray}\label{eq:u-def}
\begin{aligned}
&\osp(1|2n)^{(1)}: \quad u_{i,j,r}=\begin{cases}
 q_n^{4r}-q_n^{-4r}-q_n^{2r}+q_n^{-2r},   & \text{if } i=j=n,\\
 q_i^{r a_{ij}}- q_i^{-r a_{ij}},                             & \text {otherwise };
            \end{cases}\\
&\osp(2|2n)^{(2)}: \quad u_{i,j,r}=\begin{cases}
  (-1)^r(q_n^{2r}-q_n^{-2r}),                 & \text{if }  i=j=n,  \\
(1+(-1)^r)(q_i^{r a_{ij}/2}-q_i^{-r a_{ij}/2}),        & \text {otherwise };
                  \end{cases}\\
&\Sl(1|2n)^{(2)}: \quad \phantom{X}  u_{i,j,r}=\begin{cases}
(-1)^r(q_n^{2r}-q_n^{-2r}),                   & \text{if }  i=j=n,  \\
q_i^{r a_{ij}}- q_i^{-r a_{ij}},        & \text {otherwise }.
                 \end{cases}
\end{aligned}
\end{eqnarray}
\end{prop}

\begin{rem}
The defining relations given above are super analogues of those  in the Drinfeld realisations of quantum affine algebras given in \cite{De}.  The relations in  \cite{De} have slightly different form, but are equivalent to the original relations given by Drinfeld \cite{Dr}. They are more convenient to use for proving the equivalence of the Drinfeld-Jimbo presentation and Drinfeld realisation.
\end{rem}

For the purpose of studying vertex operator representations, it is more convenient to present the Drinfeld realisation in terms of
currents.  For this, we will need the calculus of formal distributions familiar in the theory of vertex operators algebras.  Particularly useful is the formal distribution
$\delta(z)=\sum_{r\in\Z}z^r$, which has the following property: for any formal distribution $f(z,w)$ in the two variables $z$ and $w$, we have
$f(z,w)\delta(\frac{w}{z})=f(z,z)\delta(\frac{w}{z})$. A detailed treatment of formal distributions can be found in, e.g., \cite{K98}.

Given any pair of simple roots $\alpha_i$ and $\alpha_j$ of $\g$, we let
\begin{align}\label{eq-g}
g_{ij}(z)=\sum_{n\ge 0} g_{ij,n}z^n,
\end{align}
be the Taylor series expansion at $z=0$ of $f_{ij}(z)/h_{ij}(z)$,
where
\begin{eqnarray*}
&\osp(1|2n)^{(1)}: &f_{ij}(z)=\begin{cases}
(q^{2(\alpha_i,\alpha_j)}z-1)(q^{-(\alpha_i,\alpha_j)}z-1), & i=j=n,\\
q^{(\alpha_i,\alpha_j)}z-1, & otherwise;
                                          \end{cases}\\
&                           &h_{ij}(z)=\begin{cases}
(z-q^{2(\alpha_i,\alpha_j)})(z-q^{-(\alpha_i,\alpha_j)}),&i=j=n,\\
z-q^{(\alpha_i,\alpha_j)},& otherwise;
                                           \end{cases}\\
&\osp(2|2n)^{(2)}:&f_{ij}(z)=\begin{cases}
(-q)^{(\alpha_i,\alpha_j)}z-1, & i=j=n,\\
\left(q^{(\alpha_i,\alpha_j)/2}z-1\right)\left ((-q)^{(\alpha_i,\alpha_j)/2}z-1\right), & otherwise;
                                          \end{cases}\\
&                           &h_{ij}(z)=\begin{cases}
z-(-q)^{(\alpha_i,\alpha_j)},&i=j=n,\\
\left(z-q^{(\alpha_i,\alpha_j)/2}\right)\left(z-(-q)^{(\alpha_i,\alpha_j)/2}\right),& otherwise;
                                           \end{cases}\\
&\Sl(1|2n)^{(2)}: &f_{ij}(z)=\begin{cases}
(-q)^{(\alpha_i,\alpha_j)}z-1, & i=j=n,\\
q^{(\alpha_i,\alpha_j)}z-1, & otherwise;
                                          \end{cases}\\
&                           &h_{ij}(z)=\begin{cases}
z-(-q)^{(\alpha_i,\alpha_j)},&i=j=n,\\
z-q^{(\alpha_i,\alpha_j)},& otherwise.
                                           \end{cases}
\end{eqnarray*}

Now we introduce the following formal distributions in $\Uq(\g)[[z^{1/2}, z^{-1/2}]]$ for $\g=\osp(1|2n)^{(1)}$ and $\Uq(\g)[[z, z^{-1}]]$ for $\g=\Sl(1|2n)^{(2)},\osp(2|2n)^{(2)}$,
\begin{align*}
&\e_i(z)=\begin{cases}
\sum_{r\in\Z}\e_{i,r}z^{-r+1/2}, & \g=\osp(1|2n)^{(1)},i=n;\\
\sum_{r\in\Z}\e_{i,r}z^{-r},& \text{otherwise},
\end{cases}\\
&\f_i(z)=\begin{cases}
\sum_{r\in\Z}\f_{i,r}z^{-r-1/2},& \g=\osp(1|2n)^{(1)},i=n;\\
\sum_{r\in\Z}\f_{i,r}z^{-r},& \text{otherwise},
\end{cases}
\\
&\psi_i(z)=\sum_{r\in\Z_{\ge 0}}\hh^{+}_{i,r}z^{-r},\quad \varphi_i(z)=\sum_{r\in\Z_{\le 0}}\hh^{-}_{i,r}z^{-r}.
\end{align*}

\begin{lem}\label{lem:dr-f} Let
$\g=\osp(1|2n)^{(1)}$, $\Sl(1|2n)^{(2)}$ or $\osp(2|2n)^{(2)}$. Then $\Uq(\g)$ has the following presentation. The generators are
\[
\ef_{i,r},\hh_{i,r}^{\pm}, \ \ga^{\pm 1/2},
\quad \text{for }\  (i,r)\in\mathcal{I}_{\g};
\]
the relations in terms of formal distributions are given by:
\begin{item}
\item[\rm (1)]  $\ga^{\pm 1/2}$ are central with $\ga^{1/2} \ga^{- 1/2}=1$,
\begin{align}
&\hh^{+}_{i,0}\hh^{-}_{i,0}= \hh^{-}_{i,0}\hh^{+}_{i,0}=1,
[\varphi_i(z), \psi_j(w)]=[\psi_j(w),\varphi_i(z)]=0,\\
& \varphi_i(z) \psi_j(w) \varphi_i(z)^{-1} \psi_j(w)^{-1}=g_{ij}(zw^{-1}\ga^{-1})/g_{ij}(zw^{-1}\ga),\\
& \varphi_i(z) \ef_j(w) \varphi_i(z)^{-1}=g_{ij}(zw^{-1}\ga^{\mp1/2})^{\pm1}\ef_j(w),\\
&\psi_i(z) \ef_j(w) \psi_i(z)^{-1}=g_{ij}(z^{-1}w\ga^{\mp1/2})^{\mp1}\ef_j(w),\\
&[\e_i(z), \f_j(w)]=
   \frac{\rho_{z,w}\delta_{i,j}}{q_i-q_i^{-1}}
    \left (\psi_i(z\ga^{-1/2}) \delta\left( \frac{z\ga^{-1}}{w}\right)
      -\varphi_i(z\ga^{1/2}) \delta\left(\frac{z\ga}{w}\right) \right),\label{eq:xx-f}
\end{align}
where $g_{ij}$ are defined by \eqref{eq-g}, and
$\rho_{z,w}=(z/w)^{1/2}$ if $\g=\osp(1|2n)^{(1)}$ and $i=n$,  and $\rho_{z,w}=1$ otherwise.

\item[\rm (2)] Serre relations

\begin{itemize}
\item[\rm (A)] $(i,j)\neq (n, n)$, and if $\g\neq \osp(1|2n)^{(1)}$,
 \begin{align}
\label{eq:xrs-xsr-f}
&[z^{\pm\theta}\ef_{i}(z),\ef_{j}(w)]_{q_{i}^{a_{ij}}}+[w^{\pm\theta}\ef_{j}(w),\ef_{i}(z)]_{q_{j}^{a_{ji}}}=0,
\end{align}
where $\theta=2$ if $\g=\osp(2|2n)^{(2)}$,  and $1$ if $\g=\Sl(1|2n)^{(2)}$;

\item[\rm (B)] $n\ne i\neq j$, \ $\ell=1-a_{i j}$,
\[\begin{aligned}
\hspace{10mm}
sym_{z_1,\dots,z_\ell}\sum_{k=0}^\ell  (-1)^k
                                             \begin{bmatrix} \ell\\k\end{bmatrix}_{q_i}
                                              \ef_{i}(z_1)\dots\ef_{i}(z_k) \ef_{j}(w)\ef_{i}(z_{k+1})\dots\ef_{i}(z_{\ell})=0;
\end{aligned}\]

\item[\rm (C)]  $n=i\ne j$, \ $\ell=1-a_{i j}$, and if $\g\ne \Sl(1|2n)^{(2)}$, $j<n-1$,
\[\begin{aligned}
\hspace{10mm}
sym_{z_1,\dots,z_\ell}\sum_{k=0}^\ell
                                              \begin{bmatrix} \ell\\k\end{bmatrix}_{\tilde{q_i}}
                                              \ef_{i}(z_1)\dots\ef_{i}(z_k) \ef_{j}(w)\ef_{i}(z_{k+1})\dots\ef_{i}(z_\ell)=0,
\end{aligned}\]
where $\tilde{q_i}=(-1)^{1/2}q_i$;

\item[\rm (D)] for $\g=\osp(1|2n)^{(1)}$,
 \begin{align*}
&sym_{z_1,z_2,z_3} \left[  [z_1^{\pm}\ef_{n}(z_1),\ef_{n}(z_2)]_{q_n^2},  \ef_{n}(z_3)\right]_{q_n^4}=0;\\
&sym_{z,w}\Big([z^{\pm 2}\ef_{n}(z), \ef_{n}(w)]_{q_n^2} -q_n^4 [z^{\pm}\ef_{n}(z),  w^{\pm 1}\ef_{n}(w) ]_{q_n^{-6}}\Big)=0;\\
&sym_{z_1,z_2}\Big(q_n^2\left[[z_1^{\pm}\ef_{n}(z_1),\ef_{n}(z_2)]_{q_n^2},\ef_{n-1}(w)\right]_{q_n^4}\\
&+(q_n^2+q_n^{-2})\left[[\ef_{n-1}(w),z_1^{\pm}\ef_{n}(z_1)]_{q_n^2},\ef_{n}(z_2)\right]\Big)=0;
\end{align*}

\item[\rm (E)] for $\g=\osp(2|2n)^{(2)}$,
\[
sym_{z_1,z_2} \left[  [\ef_{n-1}(w),z_1^{\pm}\ef_{n}(z_1)]_{q_n^2},  \ef_{n}(z_2)\right]=0.
\]
\end{itemize}
\end{item}
\end{lem}
\begin{proof}
This can be proven by straightforward computation, thus we will not give the details. Instead, we consider only \eqref{eq:xx-f} as an example. The relations \eqref{eq:xx} and \eqref{eq:hh-hat} lead to
\begin{align*}
&[\e_i(z), \f_j(w)]
   =\rho_{z,w}\delta_{i,j}\sum_{r,s}   \dfrac {   \ga^{\frac{r-s}{2}}  \hh^{+}_{i,r+s}
                                     -\ga^{\frac{s-r}{2}}  \hh^{-}_{i,r+s}     }
                                  {  q_i -q_i ^{-1}  }
       z^{-r}w^{-s}\\
&=\frac{\rho_{z,w}\delta_{i,j}}{q_i-q_i^{-1}}\left\{
           \sum_{r}\hh^{+}_{i,r}       (z\ga^{-1/2})^{-r}    \delta\left(\frac{z\ga^{-1}}{w}\right)
          -\sum_{r}\hh^{-}_{i,r}  (z\ga^{1/2})^{-r}    \delta\left(\frac{z\ga}{w}\right)
                                        \right\}\\
&=\frac{\rho_{z,w}\delta_{i,j}}{q_i-q_i^{-1}}\left\{
          k_i  \xp\left(
                      (q_i -q_i^{-1}) \sum_{r=1}^{\infty}\h_{i,r}(z\ga^{-1/2})^{-r}
                       \right)
       \delta \left( \frac{z\ga^{-1}}{w} \right) \right.\\
&\hspace{21mm} \left.
      -k_i^{-1} \xp \left(
                       (q_i^{-1}-q_i) \sum_{r=1}^{\infty}\h_{i,-r}(z\ga^{1/2})^{r}
                       \right)
         \delta\left(\frac{z\ga}{w}\right)          \right\}.
\end{align*}
Using the definitions of $\psi_i(z)$ and $\varphi_i(z)$ on the right hand side, we immediately obtain \eqref{eq:xx-f}. In the opposite direction, we easily obtain
\eqref{eq:xx} and \eqref{eq:hh-hat} by comparing the coefficients
of $z^{-r}w^{-s}$ in \eqref{eq:xx-f}.
\end{proof}

\subsection{Some general facts} We discuss some simple facts, which  will be used in later sections.
\subsubsection{Triangular decompositions}
We describe two triangular decompositions for the quantum affine superalgebra $\Uq(\g)$, which will be used later.

The standard triangular decomposition is
\begin{eqnarray}\label{eq:triangular-1}
\begin{aligned}
&\Uq(\g)=\U_q^{(-)} \U_q^{(0)} \U_q^{(+)},  \quad {with}\\
& \U_q^{(+)} \text{  generated by  } \e_{i,0},  \e_{i,r}, \f_{i,r},   \hh_{i,r}^{\pm}, \ \text{for $r>0$, \ $1\le i\le n$}, \\
& \U_q^{(0)} \text{  generated by  } \ga_i^{\pm 1}, \ \ga^{\pm 1/2}, \ \text{for $1\le i\le n$}, \\
& \U_q^{(-)} \text{  generated by   } \f_{i,0},   \f_{i,r}, \e_{i,r},  \hh_{i,r}^{\pm}, \ \text{for $r<0$, \ $1\le i\le n$},
\end{aligned}
\end{eqnarray}
where  $\U_q^{(-)}$, $\U_q^{(0)}$ and $\U_q^{(+)}$ are all super subalgebras of $\Uq(\g)$.
In terms of the Chevalley generators in Definition \ref{defi:quantum-super},  $\U_q^{(+)}$, $\U_q^{(-)}$ and $\U_q^{(0)}$ are respectively generated by the elements $e_j$, $f_j$ and $k^{\pm 1}_j$ with $0\le j\le n$.

The other triangular decomposition is
\begin{eqnarray}\label{eq:triangular-2}
\begin{aligned}&\Uq(\g)=\U_q^- \U_q^0 \U_q^+, \quad \text{with}\\
& \U_q^+ \text{  generated by  } \e_{i,r}, \  \text{for $1\le i\le n$, \  $r\in\Z$},\\
& \U_q^0 \text{  generated by   } \hh_{i,r}^{\pm}, \ \ga^{\pm 1/2}, \  \text{for $1\le i\le n$,  \  $r\in\Z$},  \\
& \U_q^- \text{  generated by   } \f_{i,r}, \  \text{for $1\le i\le n$,  \  $r\in\Z$},
\end{aligned}
\end{eqnarray}
where $\U_q^{-}$, $\U_q^{0}$ and $\U_q^{+}$ are also super subalgebras.
The existence of this triangular decomposition is easy to see from the Drinfeld realisation, but very obscure from the point of view of Definition \ref{defi:quantum-super}.

Let $B_q:=\U_q^{(0)} \U_q^{(+)}$ or $B_q:=\U_q^0 \U_q^+$ depending on the triangular decomposition.  A vector $v_0$ in a $\Uq(\g)$-module is a highest weight vector if $\C(q) v_0$ is a $1$-dimensional $B_q$-module.  A $\Uq(\g)$-module generated by a highest weight vector is a highest weight  module with respect to the given triangular decomposition. We will study highest weight representations with respect to both triangular decompositions in later sections.

\subsubsection{Comments on spinoral type modules} \label{sect:type-1}
One can easily see that there exist the following superalgebra automorphisms of $\Uq(\g)$.
\begin{align}\label{eq:auto-1}
\iota_\varepsilon: k_i\mapsto \varepsilon_i k_i,\quad e_i\mapsto \varepsilon_i e_i,\quad f_i\mapsto f_i, \quad 0\le i\le n,
\end{align}
for any given $\varepsilon_i\in\{\pm 1\}$.
If $V$ is $\Uq(\g)$-module, we can twist it by $\iota_\varepsilon$ to obtain another module with the same underlying vector superspace but the twisted $\Uq(\g)$-action $\Uq(\g)\otimes V\longrightarrow V$ defined by
$
x\otimes v \mapsto  \iota_\varepsilon(x)v$ for all $x\in \Uq(\g) $ and $v\in V$.
If $k_i$ $(i=0,1,\dots n)$ act semi-simply on $V$, the eigenvalues of $k_i$ are multiplied by $\varepsilon_i$ in the twisted module.

Recall the notion of type-{\bf {1}} modules in the theory of ordinary quantum groups and quantum affine algebras.  For quantum supergroups and quantum affine superalgebras, a type-{\bf {1}} module over $\Uq(\g)$ is one such that the $k_i$ $(i=0,1,\dots n)$ act semi-simply with eigenvalues of the form $q_i^m$ for $m\in\Z$.
Any weight module over an ordinary quantum group or quantum affine algebra can be twisted into a type-{\bf {1}} module by analogues of the automorphisms \eqref{eq:auto-1}.  However, that is no longer true in the present context.
As we will see from Theorem \ref{theo-finite module}, some finite dimensional simple $\Uq(\g)$-modules have $k_n$-eigenvalues of the form $\pm\sqrt{-1}q^{m+1/2}$ with $m\in\Z$.  It is not possible to twist such modules into type-{\bf 1} by the automorphisms \eqref{eq:auto-1}.

For easy reference, we introduce the following definition.
\begin{defi}\label{def:type-s}
Call a $\Uq(\g)$-module type-{\bf {s}}, meaning spinoral type,  if all $k^{\pm1}_i$ act semi-simply with eigenvalues of the following form. If  $\g=\osp(1|2n)^{(1)}$ or $\Sl(1|2n)^{(2)}$, the eigenvalues of
$k_i$ for $0\le i< n$ belong to $\{q^j\mid j\in\Z\}$, and eigenvalues of $k_n$ to $\{\sqrt{-1}q^{j+1/2}\mid j\in\Z\}$.
If  $\g=\osp(2|2n)^{(2)}$, the eigenvalues of either
$k_0$, $k_n$, or both belong to $\{\sqrt{-1}q^{j+1/2}\mid j\in\Z\}$,
and the eigenvalues of the other $k_i$ to
$\{q^j\mid j\in\Z\}$.
\end{defi}

Type-{\bf s} modules exist  even for the quantum supergroup $\Uq(\osp(1|2))$ associated with $\osp(1|2)$.

\begin{example}[Type-{\bf s} representations of $\Uq(\osp(1|2))$]
The quantum supergroup $\Uq(\osp(1|2))$ is generated by $E, F$ and $K^{\pm1}$ with relations $K K^{-1}=K^{-1} K=1$ and
\[
K E K^{-1} = q E, \quad K F K^{-1} = q^{-1} F, \quad E F + F E = \frac{K-K^{-1}}{q-q^{-1}}.
\]
It has long been known that there exists an $\ell$-dimensional irreducible representation of $\Uq(\osp(1|2))$ for each positive integer $\ell$.  If $\ell$ is odd, the irreducible representation can be twisted into a type-{\bf 1} representation;  and  if $\ell$ is even, to a type-{\bf s} representation.

The smallest type-{\bf s} example is the $2$-dimensional irreducible representation, which is given by
\[
E\mapsto \begin{pmatrix}
0&\frac{\sqrt{-1}}{q^{1/2}-q^{-1/2}}\\
0&0 \end{pmatrix}, \quad
F\mapsto \begin{pmatrix} 0&0\\ 1&0 \end{pmatrix}, \quad  K\mapsto\begin{pmatrix}\sqrt{-1}q^{1/2}&0\\0&\sqrt{-1}q^{-1/2}\end{pmatrix}.
\]
\begin{rem}
The $2$-dimensional irreducible representation  of
$\Uq(\osp(1|2))$ does not have a classical limit, i.e., $q\to 1$ limit, nor do all the even dimensional irreducible representations.  This agrees with the fact that the finite dimensional irreducible representations of $\osp(1|2)$ are all odd dimensional.
\end{rem}
\end{example}

The quantum affine superalgebra $\Uq(\g)$ for all $\g$ in \eqref{eq:g} contains the quantum supergroup $\Uq(\osp(1|2))$ as a super subalgebra.  The type-{\bf s} representations of $\Uq(\g)$ restrict to type-{\bf s} representations of $\Uq(\osp(1|2))$.


\section{Vertex operator  representations}\label{sect:vertex}
We construct vertex  operator representations of the quantum affine superalgebras $\Uq(\g)$ for all $\g$ in \eqref{eq:g}. These representations are level $1$
irreducible integrable highest weight representations relative to the standard triangular
decomposition \eqref{eq:triangular-1} of  $\Uq(\g)$.  By level $1$ representations, we mean those with $\gamma$ acting by multiplication by $\pm q$ or  $\sqrt{-1}q$.

Our construction involves generalising to the quantum affine superalgebra context some aspects of  \cite{LP}.
The vertex operators obtained here have considerable similarities with those  \cite{Jn1, JnM} for ordinary twisted quantum affine algebras.

\subsection{The Fock space}\label{sec:space}
Let $\ell(\alpha_i):=(\alpha_i,\alpha_i)$ for any simple root $\alpha_i$.
For convenience,
we choose the normalisation for the bilinear form so that $\ell(\alpha_n)=2$ if $\g=\osp(2|2n)^{(2)}$,  and $\ell(\alpha_n)=1$ otherwise. Let $\wp=(-1)^{1/\ell(\alpha_n)}q$, and take $\wp^{1/2}=(-1)^{\frac{1}{2\ell(\alpha_n)}}q^{1/2}$.

Hereafter we will always consider $\Uq(\g)$ in the Drinfeld realisation  given in Proposition \ref{prop:dr} and Lemma \ref{lem:dr-f}.
Denote by $\Uq(\widetilde{\eta})$
the subalgebra of $\Uq(\g)$ generated by the elements  $\gamma^{1/2}$, $\qk_i$ and $\h_{i,r}$ ($r\in\Z\backslash{\{0\}}$, $1\le i\le n$),
and by $\U_q(\eta)$ that generated $\gamma^{1/2}$ and $\h_{i,r}$  ($r\in\Z\backslash{\{0\}}$, $1\le i\le n$).
Let $S(\eta^{-})$ be the symmetric algebra generated by $\h_{i,r}$ for $r\in\Z_{<0}$ and $1\le i\le n$.  Let $H_i(s)$ ($s\in\Z\backslash\{0\}$, $1\le i\le n$) be the linear operators acting on $S(\eta^{-})$ such that
\begin{eqnarray}\label{eq:vo-H}
\begin{aligned}
&\HH_i(-s)=\text{derivation defined by}\\
&\phantom{HH_i(-s)} \HH_i(-s)(\h_{j,r})=\delta_{r,s} \dfrac{ u_{i,j,-s} (\wp^{s}-\wp^{-s})  }
                                                           { s   (q_i-q_i^{-1})(q_j-q_j^{-1})  }, \\
&\HH_i(s)=\text{multiplication by $\h_{i,s}$}, \qquad \forall r, s\in\Z_{<0},
\end{aligned}
\end{eqnarray}
where $u_{i,j,-s}$ is defined by \eqref{eq:u-def}. Then
\begin{align}\label{eq:vo-hh}
[\HH_{i}(r),\HH_{j}(s)]=\delta_{r+s,0}\dfrac{u_{i,j,r} (\wp^{r}-\wp^{-r})  }
                                  { r   (q_i-q_i^{-1})(q_j-q_j^{-1})  },
                                 \quad \forall r,s\in\Z\backslash\{0\}.
\end{align}
The algebra  $\U_q(\eta)$ has the canonical irreducible representation on $S(\eta^{-})$ given by
\[
\begin{aligned}
\ga \mapsto \wp, \quad \h_{i,s} \mapsto \HH_i(s), \quad \forall   s\in\Z\backslash\{0\}.
\end{aligned}
\]

Let $\dot{\g}\subset\g$ be the regular simple Lie sub-superalgebra with the Dynkin diagram obtained from the Dynkin diagram of $\g$ by removing the node corresponding to $\alpha_0$. Then $\dot{\g}=\osp(1|2n)$ in all three cases of $\g$. Let $\Q$ be the root lattice of $\dot{\g}$ with the bilinear form inherited from that of $\g$. We regard $\Q$ as a multiplicative group consisting of elements of the form $e^\alpha$ with $\alpha\in \Q$. Let $\C[\Q]$ be the group algebra of $\Q$. Given any variable $z$ and any root $\alpha$,  we define a linear operator on $\C[\Q]$ by
\begin{eqnarray}\label{eq:operator-on-CQ}
 z^{\alpha}. e^{\beta}=z^{(\alpha,\beta)}e^{\beta}.
\end{eqnarray}
We also define the linear operator $\sigma_i$ on $\C[\Q]$ for all $i=1, 2, \dots, n$  by
\[\begin{aligned}
\sigma_i. e^{\beta}=(-1)^{(\alpha_i,\beta)}e^{\beta}.
\end{aligned}\]
Write $\Phi_i=\prod_{k=i}^{n}\sigma_k$ for $1\le i\le n$ and $\Phi_i=1$ for $i>n$. It is easy to check that $\Phi_i.e^{\pm\alpha_j}=(-1)^{\delta_{i,j}+\delta_{i+1,j}}e^{\pm\alpha_j}$ for $1\le i,j \le n$ and $\Phi_i^2=1$.

We also need some basic knowledge of the $q$-deformed Clifford algebra $\cq$,
which is generated by $\ka(r),\ka(s)$ ($r,s\in\Z +\frac{1}{2}$) with relations
\begin{equation}\label{eq:kk}
\ka(r)\ka(s)+\ka(s)\ka(r)=\delta_{r,-s}(q^r+q^{s}), \quad \forall r, s.
\end{equation}
We use $\Lambda(\cq^{-})$ to denote the exterior algebra generated by $\ka(r)$ for $r<0$,  and denote by $\Lambda(\cq^{-})_0$ (resp. $\Lambda(\cq^{-})_1$) the subspace of even (resp. odd) degree, where $\ka(r)$ ($r<0$) are regarded as having degree $1$.  Define the linear operators $K (s)$ on $\Lambda(\cq^{-})$ such that for any $\psi, \phi\in \Lambda(\cq^{-})$,
\[
\begin{aligned}
&K(s)\cdot\psi = \ka(s)\psi, \quad K (-s)\cdot\ka(r)=\delta_{r,s}(q^r+q^{-r}), \quad K (-s)\cdot 1=0, \\
&K(-s)\cdot(\psi\phi)= K(-s)\cdot(\psi) \phi + (-1)^{deg(\psi)}\psi K(-s)\cdot(\phi), \quad \forall r, s<0.
\end{aligned}
\]
Then $\cq$ acts on $\Lambda(\cq^{-})$ by
$\ka(r)\mapsto K(r)$ for all $r\in\Z+\frac{1}{2}$.
Let
\begin{align}\label{eq:v}
W=\left\{
\begin{aligned}
&\C[\Q],  \quad
                      \g=\osp(1|2n)^{(1)},\osp(2|2n)^{(2)};\\
&\C[\Q_0]\otimes \Lambda(\mathcal{C}_{\wp}^{-})_0  \oplus \C[\Q_0]e^{\lambda_1}\otimes \Lambda(\mathcal{C}_{\wp}^{-})_1,   \quad
                      \g=\Sl(1|2n)^{(2)},
\end{aligned} \right.
\end{align}
where $\Q_0$ is the lattice spanned by the set of roots with squared length 2 and $\lambda_1=\alpha_1+\alpha_2+\dots+\alpha_n$.
Now we construct the vector space
$
V=S(\eta^{-})\otimes W.
$


\subsection{Construction of the vacuum representations}\label{sec-vo-0}
We start by defining
\begin{align*}
&K(z)=\sum_{s\in\Z +1/2}K(s)z^{-s}, \\
&T^{+}_i(z)=\begin{cases}
e^{\alpha_i} \Phi_i z^{\alpha_i+\ell(\alpha_i)/2}, \quad \text{if  }  \g=\osp(1|2n)^{(1)}, \osp(2|2n)^{(2)};\\
e^{\alpha_i}  \Phi_i  z^{\alpha_i+\ell(\alpha_i)/2}K(z), \quad \text{if  }  \g=\Sl(1|2n)^{(2)},
\end{cases}\\
&T^{-}_i(z)=\begin{cases}
e^{-\alpha_i} \Phi_{i+1} z^{-\alpha_i+\ell(\alpha_i)/2}, \quad \text{if  } \g=\osp(1|2n)^{(1)}, \osp(2|2n)^{(2)};\\
e^{-\alpha_i}  \Phi_{i+1}  z^{-\alpha_i+\ell(\alpha_i)/2}(- K(z)), \quad \text{if  }  \g=\Sl(1|2n)^{(2)},
\end{cases}
\end{align*}
and introducing the following formal distributions:
\begin{align*}
&E^{\pm}_{i}(z) =\xp\left(
        \pm\sum_{k=1}^{\infty}\frac{\wp^{\mp k/2}}{ \{k\}_{q_i} }\HH_i(-k)z^k
                                  \right),\\
&F^{\pm}_{i}(z) =\xp\left(
       \mp\sum_{k=1}^{\infty}\frac{\wp^{\mp k/2}}{ \{k\}_{q_i}}\HH_i(k)z^{-k}
                                 \right),
\end{align*}
where $\{k\}_{q_i}=[k]_{\wp}\cdot \frac{\wp-\wp^{-1}}{q_i-q_i^{-1}}=\frac{\wp^k-\wp^{-k}}{q_i-q_i^{-1}}$.
Using them, we define linear operators $\X^{\pm}_j(k)$ ($1\le j\le n$,   $k\in\Z$) on the vector space $V$ by
\begin{align}\label{eq:vo}
\X^{\pm}_{i}(z)=E^{\pm}_{i}(z)F^{\pm}_{i}(z)T^{\pm}_i(z), \quad i=1, 2, \dots, n,
\end{align}
where
\[
\begin{aligned}
\X^{\pm}_i(z)&=\sum_{k\in\Z} \X^{\pm}_i(k) z^{-k}, \ \qquad \text{for all $i\ne n$,}\\
\X^{\pm}_n(z)&=\sum_{k\in\Z} \X^{\pm}_n(k) z^{-k}, \ \qquad \text{if  $\g\ne\osp(1|2n)^{(1)}$,}\\
\X^{\pm}_n(z)&=\sum_{k\in\Z} \X^{\pm}_n(k) z^{-k+1/2},  \quad \text{ if $\g=\osp(1|2n)^{(1)}$}.
\end{aligned}
\]

We have the following result.
\begin{theo}\label{them:v.o}
The quantum affine superalgebra $\Uq(\g)$ acts irreducibly
on the vector space  $V$, with the action defined by
\begin{eqnarray}\label{eq:action}
\begin{aligned}
&\ga^{1/2}\mapsto \wp^{1/2},\ \ \qk_i^{1/2}\mapsto (\varpi_i\sigma_i \wp^{\alpha_i})^{1/2},
\ \  \h_{i,s}\mapsto \HH_i(s),\\
&\e_{i,k}\mapsto \X^{+}_i(k),\ \ \f_{i,k}\mapsto \varrho_i\X^{-}_i(k),\\
&\forall i=1, 2, \dots, n, \ \ s\in\Z\backslash\{0\}, \ \ k\in\Z,
\end{aligned}
\end{eqnarray}
where
\[\begin{aligned}
& \varpi_i=\begin{cases}
  \wp^{-1/2}, &\text{if \  } \g=\osp(1|2n)^{(1)}, \ i= n;\\
   1,               & \text{otherwise};
             \end{cases} \\
&\varrho_i=\begin{cases}
 -2^{-1}\{\ell(\alpha_i)/2\}_{q_i}, &\text{if \  } \g=\osp(2|2n)^{(2)}, \ i\neq n;\\
 -\{\ell(\alpha_i)/2\}_{q_i},           & \text{otherwise}.
                \end{cases}
\end{aligned}\]
\end{theo}

\begin{proof}
The irreducibility of $V$ as a $\Uq(\g)$-module follows from the fact that the $\Uq(\eta)$-module $S(\eta^{-})$ and $\Uq(\widetilde{\eta})$-module  $W$ are both irreducible. This was proved in \cite{Jn96}.
Thus the proof of the theorem essentially boils down to verifying that the operators
$ \HH_i(k)$ and $\X^{\pm}_i(k)$ satisfy the commutation relations of
$\h_{i,k}$ and $\ef_{i,k}$. We show this by using  calculus of
 formal distributions.

Consider the vertex operators \eqref{eq:vo} in the case of  $\g=\osp(1|2n)^{(1)}$.
We claim that they satisfy the following relation (cf. \eqref{eq:xx-f}):
\begin{eqnarray}\label{eq:vo-XX}
\begin{aligned}
&[\X^{+}_{i}(z), \X^{-}_{j}(w)]
=\frac{\delta_{i j}\,\rho_{z,w}\,\varrho^{-1}_i}{q_i-q_i^{-1}}
     \left\{\sigma_i \varpi_i \wp^{\alpha_i} \widetilde{V}_i^+(z)
                  \delta\left(\wp^{-1}\frac{z}{w}\right)
      \right.\\
&\hspace{28mm} \left.
     -\sigma_i \varpi_i \wp^{-\alpha_i}
\widetilde{V}_i^-(z)
       \delta\left(\wp\frac{z}{w}\right) \right\},
\end{aligned}
\end{eqnarray}
where
\begin{eqnarray}\label{eq:widetildeV}
\begin{aligned}
&\widetilde{V}_i^+(z)=\xp\left(
            \sum_{k=1}^{\infty}   (q_i-q_i^{-1})    \HH_i(k)(z\wp^{-1/2})^{-k}     \right), \\
&\widetilde{V}_i^-(z)=\xp\left( \sum_{k=1}^{\infty}(q_i^{-1}-q_i)\HH_i(-k)(z\wp^{1/2})^{k}\right).
\end{aligned}
\end{eqnarray}

If $(\alpha_i,\alpha_j)=0$ (necessarily $i\ne j$), the claim is clear.

If  $(\alpha_i,\alpha_j)\neq 0$, there are three possibilities:
$(\alpha_i,\alpha_j)=-1$ with $i\ne j$,  and $(\alpha_i,\alpha_j)=1$ or $2$ with $i= j$.

Define normal ordering as usual by placing $\HH_i(-k)$ with $k>0$ on the left of
$\HH_j(k)$, $\exp^\alpha$ on the left of $z^\beta$, and $K(-s)$ with $s>0$ on the left of $K(s)$, where for the $K(r)$'s an order change procures a sign.
Let
\begin{align*}
:T^{+}_i(z)T^{-}_j(w):&=e^{ \alpha_i-\alpha_j} \Phi_i\Phi_{j+1}  z^{ \alpha_i}w^{-\alpha_j},
\end{align*}
then we have the following relations:  if $(\alpha_i,\alpha_j)\neq 1$,
\begin{align*}
:T_i^+(z)T_j^-(w):&=(-1)^{\delta_{i-1,j}+\delta_{i,j}}  T_i^+(z)T_j^-(w)  z^{(\alpha_i,\alpha_j)}   z^{-\ell(\alpha_i)/2}   w^{-\ell(\alpha_j)/2}\\
&=(-1)^{\delta_{i-1,j}+\delta_{i,j}}  T_j^-(w)T_i^+(z) w^{(\alpha_i,\alpha_j)} z^{-\ell(\alpha_i)/2}  w^{-\ell(\alpha_j)/2};
\end{align*}
if $(\alpha_i,\alpha_j)= 1$,
\begin{align*}
:T_i^+(z)T_j^-(w):&=-T_i^+(z)T_j^-(w)  z^{(\alpha_i,\alpha_j)} z^{-\ell(\alpha_i)/2}  w^{-\ell(\alpha_j)/2}\\
&=T_j^-(w)T_i^+(z) w^{(\alpha_i,\alpha_j)} z^{-\ell(\alpha_i)/2}  w^{-\ell(\alpha_j)/2}.
\end{align*}
Also
\begin{eqnarray}\label{eq:normal order-X}
\begin{aligned}
:\X^{+}_{i}(z)\X^{-}_{j}(w):=E^{+}_{i}(z)E^{\pm}_j(w)F^{+}_{i}(z)F^{-}_{j}(w):T^{+}_i(z)T^{-}_j(w):.
\end{aligned}
\end{eqnarray}
Thus $\X^{-}_{j}(w)\X^{+}_{i}(z)$ can be expressed as
\begin{align*}
:\X^{+}_{i}(z)\X^{-}_{j}(w):  \xp \left( \sum_{k=1}^{\infty} \frac{u_{i,j,k}}{k(\wp^k-\wp^{-k})}z^{-k}w^k \right)z^{(\alpha_i,-\alpha_j)}z^{\ell(\alpha_i)/2}w^{\ell(\alpha_j)/2},
\end{align*}
where we have used the Baker-Campbell-Hausdorff formula.

Let  $\delta_1(x)=\sum_{n\le 0}(\wp^{-n}-\wp^{n})x^n$. Then direct computation shows that
$\X^{+}_{i}(z)\X^{-}_{j}(w)$ can be expressed as
\begin{align*}
  & :\X^{+}_{i}(z)\X^{-}_{j}(w): (-1)^{\delta_{i-1,j}}  (z+w)\, z^{\ell(\alpha_i)/2}w^{\ell(\alpha_j)/2} , \quad \text{if \ }  (\alpha_i,\alpha_j)=-1,\\
  & :\X^{+}_{i}(z)\X^{-}_{j}(w): \frac{1}{\wp-\wp^{-1}}\delta_1(z/w),  \quad \text{if \ }  (\alpha_i,\alpha_j)=2,\\
  & :\X^{+}_{i}(z)\X^{-}_{j}(w): \frac{1}{\wp-\wp^{-1}} \delta_1(z/w)\, (z+w)(zw)^{-1/2}, \quad \text{if \ }  (\alpha_i,\alpha_j)=1,
\end{align*}
where we have used the formula $\text{ln}(1-x)=-\sum_{n=1}^{\infty}\frac{x^n}{n}$.
Note that $z^{\pm 1/2}$ and $w^{\pm 1/2}$ may appear in $\X^{+}_{i}(z)\X^{-}_{j}(w)$.
A similar computation shows that
\begin{itemize}
\item if $(\alpha_i,\alpha_j)=-1$,
\begin{align*}
 \X^{-}_{j}(w)\X^{+}_{i}(z)=:\X^{+}_{i}(z)\X^{-}_{j}(w): (-1)^{\delta_{i-1,j}} (z+w)\, z^{\ell(\alpha_i)/2}w^{\ell(\alpha_j)/2},
\end{align*}

\item if $ (\alpha_i,\alpha_j)=2,$
\begin{align*}
\X^{-}_{j}(w)\X^{+}_{i}(z)=:\X^{+}_{i}(z)\X^{-}_{j}(w): \frac{1}{\wp-\wp^{-1}}\,\delta_1(w/z),
\end{align*}

\item if $(\alpha_i,\alpha_j)=1, $
\begin{align*}
&&\X^{-}_{j}(w)\X^{+}_{i}(z)=:\X^{+}_{i}(z)\X^{-}_{j}(w): \frac{1}{\wp^{-1}-\wp}\, \delta_1(w/z)\, (z+w)(zw)^{-1/2}.
\end{align*}
\end{itemize}
Using these we obtain
\begin{align*}
&[\X^{+}_{i}(z),\X^{-}_{j}(w)]=\X^{+}_{i}(z) \X^{-}_{j}(w)-(-1)^{[\alpha_i][\alpha_j]}\X^{-}_{j}(w) \X^{+}_{i}(z)\\
=&\begin{cases}
:\X^{+}_{i}(z)\X^{-}_{j}(w):\frac{(z+w)(zw)^{-1/2}}{\wp-\wp^{-1}}\left(\delta(\wp^{-1}z/w)-\delta(\wp z/w)\right),
                     &(\alpha_i,\alpha_j)=1,\\
:\X^{+}_{i}(z)\X^{-}_{j}(w):\frac{1}{\wp-\wp^{-1}}\left(\delta(\wp^{-1}z/w)-\delta(\wp z/w)\right),
                     &(\alpha_i,\alpha_j)=2,\\
0,                  &(\alpha_i,\alpha_j)=-1,
\end{cases}
\end{align*}
where $[\alpha_i]=0$ if $\alpha_i$ is an even root, and 1 otherwise. This in particular shows that  \eqref{eq:vo-XX} holds  for all $i\ne j$.

In the cases with $i=j$, by using  $f(z,w)\delta(\frac{w}{z})=f(z,z)\delta(\frac{w}{z})$, we obtain
\begin{align*}
&:\X^{+}_{i}(z)\X^{-}_{j}(w): \delta\left(\wp^{-1}\frac{z}{w}\right)
=-\sigma_i \wp^{\alpha_i}
        \widetilde{V}_i^+(z)
         \delta\left(\wp^{-1}\frac{z}{w}\right), \\
&:\X^{+}_{i}(z)\X^{-}_{j}(w): \delta\left(\wp \frac{z}{w}\right)
=-\sigma_i \wp^{-\alpha_i}
\widetilde{V}_i^-(z)
       \delta\left(\wp\frac{z}{w}\right),
\end{align*}
where $\widetilde{V}_i^+(z)$ and $\widetilde{V}_i^-(z)$ are defined by \eqref{eq:widetildeV}. Note that
\[
\begin{aligned}
&(z+w)(zw)^{-1/2}\delta\left(\wp^{\pm1}\frac{z}{w}\right)=(z/w)^{1/2}(1+\wp^{\pm1})\delta\left(\wp^{\pm1}\frac{z}{w}\right).
\end{aligned}
\]
These formulae immediately lead to \eqref{eq:vo-XX}.

To consider the Serre relations, we take as an example the relation \eqref{eq:xrs-xsr-f} when $(\alpha_i,\alpha_j)=-1$. In this case, \eqref{eq:xrs-xsr-f} is equivalent to
\begin{align*}
   (z-q^{-1}w)\e_i(z)\e_j(w)
=(q^{-1}z-w)\e_j(w)\e_i(z).
\end{align*}
Thus, we need to show
\begin{align}\label{eq:XX}
 (z-q^{-1}w)\X^{+}_{i}(z)\X^{+}_{j}(w)
=(q^{-1}z-w)\X^{+}_{j}(w)\X^{+}_{i}(z).
\end{align}

Let $:T^{+}_i(z)T^{+}_j(w):=e^{ \alpha_i+\alpha_j} \Phi_i\Phi_{j} z^{ \alpha_i}w^{\pm\alpha_j},$ and
\[\begin{aligned}
:\X^{+}_{i}(z)\X^{+}_{j}(w):=E^{+}_{i}(z)E^{+}_j(w)F^{+}_{i}(z)F^{+}_{j}(w):T^{+}_i(z)T^{+}_j(w):.
\end{aligned}\]
By \eqref{eq:normal order-X}, $\X^{+}_{i}(z)\X^{+}_{j}(w)$ is equal to
\begin{align*}
&:\X^{+}_{i}(z)\X^{+}_{j}(w):
\xp\left[    -\sum_{k=1}^{\infty} \frac {   \wp^{-k}} {  \{k\}_{q_i}\{k\}_{q_j}  }  [\HH_i(k),\HH_j(-k)]
               \left( \frac{w}{z} \right)^k     \right] z^{-1} z^{\ell(\alpha_i)} w^{\ell(\alpha_j)},
               \end{align*}
which can be simplified to
 $:\X^{+}_{i}(z)\X^{+}_{j}(w): \left(1-q^{-1}\frac{w}{z}\right)^{-1} z^{-1} z^{\ell(\alpha_i)} w^{\ell(\alpha_j)}.
$
Thus
\[
\X^{+}_{i}(z)\X^{+}_{j}(w) =:\X^{+}_{i}(z)\X^{+}_{j}(w): (-1)^{\delta_{i-1,j}} \left(z-q^{-1}w\right)^{-1}  z^{\ell(\alpha_i)} w^{\ell(\alpha_j)}.
\]
Similarly we can show that
\begin{align*}
\X^{+}_{j}(w)\X^{+}_{i}(z)=:\X^{+}_{i}(z)\X^{+}_{j}(w): (-1)^{\delta_{i,j-1}} \left(w-q^{-1}z\right)^{-1} z^{\ell(\alpha_i)} w^{\ell(\alpha_j)}.
\end{align*}
Note that $i=j-1$ or $j+1$ in this case. Then two relations above immediately imply \eqref{eq:XX}.

Similar computation proves the theorem for the other $\g$.
\end{proof}

\begin{rem}
The representations in Theorem \ref{them:v.o} are not of type-{\bf 1}.
Note in particular that $\gamma$ acts by $\wp$. However, we can twist them into type-{\bf 1} or type-{\bf s} representations (see Definition \ref{def:type-s}) by the automorphisms  \eqref{eq:auto-1}.
\end{rem}


\subsection{Construction of the other level $1$ irreducible representations}\label{sect:other-level-1}
We now consider the vertex operator construction for the other level $1$ irreducible integrable highest weight representations with respect to the standard triangular
decomposition \eqref{eq:triangular-1}.
Observe that for  $\Uq(\osp(1|2n)^{(1)})$, the vacuum representation is the only such representation. Thus we will consider $\Uq(\g)$ for
$\g=\Sl(1|2n)^{(2)}$ and $\osp(2|2n)^{(2)}$ only.
We will only state the main results;  their proofs are quite similar to those in Section \ref{sec-vo-0}.

We maintain the notation of Section \ref{sec-vo-0}.

\subsubsection{The case of  $\Uq(\Sl(1|2n)^{(2)})$ }
There is only one irreducible integrable highest weight representation at level 1 beside the vacuum representation. It can be constructed as follows.

Recall the definition of $W$ in  \eqref{eq:v}.  Let $\lambda_n$ be the fundamental weight of $\dot{\g}$ corresponding to $\alpha_n$, and consider the subset
$\lambda_n+\Q$ of the weight lattice of $\dot{\g}$.  The linear operators $z^{\alpha}$ defined by \eqref{eq:operator-on-CQ} act on the group algebra of the weight lattice of
$\dot{\g}$ in the obvious way.
Denote $W_n=e^{\lambda_n}\C[\Q]$ and $V_n=S(\eta^{-})\otimes W_n$. Then $V_n$ is the level $1$ simple $\Uq(\Sl(1|2n)^{(2)})$-module with the action give by \eqref{eq:action} in terms of vertex operators. The highest weight vector is $1\otimes e^{\lambda_n}$.

\subsubsection{The case of $\Uq(\osp(2|2n)^{(2)})$ }
There are another two simple integrable highest weight modules at level $1$,  respectively associated with the fundamental weights $\lambda_1$ and $\lambda_n$ of $\dot{\g}$. Here $\lambda_1$ and $\lambda_n$ correspond to $\alpha_1$ and $\alpha_n$ respectively.  To construct the representations, we need the following q-deformed Clifford algebra $\cqq$, which is generated by $\ta(r),\ta(s)$ ($r,s\in\Z$) with relations
\begin{equation}\label{eq:tt}
\ta(r)\ta(s)+\ta(s)\ta(r)=\delta_{r,-s}(q^r+q^{s}), \quad \forall r, s.
\end{equation}
These are q-deformed Ramond fermionic operators.
Similar to Section \ref{sec:space}, we define linear operators $T (s)$ acting on $\Lambda(\cqq^{-})$ such that for any $\psi, \phi\in \Lambda(\cqq^{-})$,
\[
\begin{aligned}
&T(s)\cdot\psi = \ta(s)\psi, \quad T (-s)\cdot\ka(r)=\delta_{r,s}(q^r+q^{-r}), \quad T (-s)\cdot 1=0, \\
&T(-s)\cdot(\psi\phi)= T(-s)\cdot(\psi) \phi + (-1)^{deg(\psi)}\psi T(-s)\cdot(\phi), \quad \forall r, s<0,
\end{aligned}
\]
and $T(0)$ acts as the identity.

We replace $K(z)$ in Section \ref{sec-vo-0} by
$
K(z)=\sum_{s\in\Z}T(s)z^{-s}
$
and use it in \eqref{eq:vo} to obtain the corresponding vertex operators.
Now define
\[\begin{aligned}
&V^{(1)}=S(\eta^{-})\otimes W^{(1)} \quad \text{and} \quad V^{(n)}=S(\eta^{-})\otimes W^{(n)} \ \text{ with} \\
&W^{(1)}=e^{\lambda_1}\C[\Q_0]\otimes \Lambda(\mathcal{C}_{\wp}^{-})_0 \oplus \C[\Q_0]\otimes \Lambda(\mathcal{C}_{\wp}^{-})_1,\\
&W^{(n)}=e^{\lambda_n}\C[\Q]\otimes \Lambda(\mathfrak{C}_{\wp}^{-}).
\end{aligned}\]
Then $V^{(1)}$ and $V^{(n)}$ are the simple $\Uq(\osp(2|2n)^{(2)})$-modules at level $1$  with the actions formally given by \eqref{eq:action} but in terms of the new vertex operators.  The highest weight vectors are $1\otimes e^{\lambda_1}\otimes 1$ and $1\otimes e^{\lambda_n}\otimes 1$ respectively.


\subsection{Another construction of vacuum representations}\label{sect:other-vacuum}
For the quantum affine superalgebras
$\Uq(\osp(1|2n)^{(1)})$ and $\Uq(\Sl(1|2n)^{(2)})$, it is possible to modify the
vertex operators of the vacuum representations to make $\gamma$ act by $q$,
and this is what we will do in this section. The modified vertex operator representation of $\Uq(\Sl(1|2n)^{(2)})$ given here is of type-{\bf 1}.

For both affine superalgebras,  we choose in this section the normalisation for the bilinear form on the weight space so that $(\alpha_n,\alpha_n)=1$.

Recall the definitions of  $ \U_q(\eta)$ and $S(\eta^{-})$ in section \ref{sec:space}.  Let us now define new linear operators acting on $S(\eta^{-})$, denoted by
$H_i^q(s)$ with $s\in\Z\backslash\{0\}$, $1\le i\le n$,  as follows.
\begin{eqnarray}\label{eq:vo-Hp}
\begin{aligned}
&\HH_i^q(-s)=\text{derivation defined by}\\
&\phantom{HH_iq(-s)} \HH_i^q(-s)(\h_{j,r})=\delta_{r,s} \dfrac{ u_{i,j,-s} (q^{s}-q^{-s})  }
                                                           { s   (q_i-q_i^{-1})(q_j-q_j^{-1})  }, \\
&\HH_i^q(s)=\text{multiplication by $\h_{i,s}$}, \qquad \forall r, s\in\Z_{<0},
\end{aligned}
\end{eqnarray}
where $u_{i,j,-s}$ is defined by \eqref{eq:u-def}.  This differs from
\eqref{eq:vo-H} in that $\wp$  is replaced by $q$. Now we have
\begin{align}\label{eq:vo-hh-q}
[\HH^q_{i}(r),\HH^q_{j}(s)]=\delta_{r+s,0}\dfrac{u_{i,j,r} (q^{r}-q^{-r})  }
                                  { r   (q_i-q_i^{-1})(q_j-q_j^{-1})  },
                                 \quad \forall r,s\in\Z\backslash\{0\},
\end{align}
and we obtain the following irreducible $\U_q(\eta)$-representation on $S(\eta^{-})$
\[
\begin{aligned}
\ga \mapsto q, \quad \h_{i,s} \mapsto \HH_i(s), \quad \forall   s\in\Z\backslash\{0\}.
\end{aligned}
\]

Define the $2$-cocycle $C:\Q \times \Q \to \{\pm 1\}$, satisfying
\[
\begin{aligned}
         C(\alpha+\beta,\ga )=C(\alpha,\ga)C(\beta,\ga),
\quad C(\alpha,\beta+\ga)=C(\alpha,\beta)C(\alpha,\ga),
\quad \forall \alpha, \beta, \ga,
\end{aligned}
\]
such that $ C(0, \beta)=C(\alpha, 0)=1$, and for  any simple roots $\alpha_i$ and $\alpha_j$,
\[\begin{aligned}
C(\alpha_i,\alpha_j)=\begin{cases}
(-1)^{(\alpha_i,\alpha_j)+(\alpha_i,\alpha_i)(\alpha_j,\alpha_j)}, & i\le j,\\
1,& i>j.
\end{cases}
\end{aligned}\]
Obviously,  $\Q$ has a unique central extension $\hQ$,
\[\begin{aligned}
1\rightarrow  \Z_2 \rightarrow   \hQ  \rightarrow \Q\rightarrow 1
\end{aligned}\]
defined in the following way.  We regard $\hQ$ as a multiplicative group consisting of elements $\pm e^\alpha$ with $\alpha\in \Q$. Then $(-1)^a e^\alpha (-1)^b e^\beta=(-1)^{a+b}C(\alpha,\beta)e^{\alpha+\beta}$, where $a, b\in\{0, 1\}$ and $\alpha, \beta\in \Q$.
Let $\C[\hQ]$ be the group algebra of  $\hQ$, and let $J$ be the two-sided ideal generated by $e^\alpha +(-e^\alpha)$ for all $\alpha$. Denote the quotient $\C[\hQ]/J$ by $\C[\Q]$.
Now $\pm e^\alpha\in\C[\hQ]$ are natural linear operators acting on $\C[\Q]$. The linear operators $z^{\alpha}$ are the same as in section \ref{sec:space}.

Let $V=S(\eta^{-})\otimes W$, where $W$ is defined in \eqref{eq:v}.

For all $i=1, \dots, n$, let
\begin{align*}
\widetilde{T}^{\pm}_i(z)=\begin{cases}
e^{\pm\alpha_i} z^{\pm\alpha_i+\ell(\alpha_i)/2}, & \g=\osp(1|2n)^{(1)};\\
e^{\pm\alpha_i} z^{\pm\alpha_i+\ell(\alpha_i)/2}(\pm K(z)), & \g=\Sl(1|2n)^{(2)}.
\end{cases}
\end{align*}
Define
\[
\begin{aligned}
&\widetilde{E}^{\pm}_{i}(z)=\xp\left(
        \pm\sum_{k=1}^{\infty}\frac{q^{\mp k/2}}{ [k]_{q_i} }\HH^q_i(-k)z^k
                                  \right), \\
&\widetilde{F}^{\pm}_{i}(z)=\xp\left(
       \mp\sum_{k=1}^{\infty}\frac{q^{\mp k/2}}{ [k]_{q_i}}\HH^q_i(k)z^{-k}
                                 \right), \quad \text{for $i\ne n$}; \\
&\widetilde{E}^{\pm}_n(z)=\xp\left(
        \pm\sum_{k=1}^{\infty}\frac{q^{\mp k/2}}{ [2k]_{q_n} }\HH^q_n(-k)z^k
                                  \right), \\
&\widetilde{F}^{\pm}_n(z)=\xp\left(
       \mp\sum_{k=1}^{\infty}\frac{q^{\mp k/2}}{ [2k]_{q_n}}\HH^q_n(k)z^{-k}
                                 \right),
\end{aligned}
\]
and finally set
\begin{align}\label{eq:vo-p}
\widetilde{\X}^{\pm}_{i}(z) &=\widetilde{E}^{\pm}_{i}(z) \widetilde{F}^{\pm}_{i}(z)
                            \widetilde{T}^{\pm}_i(z), \quad \forall i.
\end{align}
%
Similar arguments as those in the proof of Theorem \ref{them:v.o} can prove the following result.
\begin{theo}\label{them:v.o-q}
Let $\g$ be $\osp(1|2n)^{(1)}$ or $\Sl(1|2n)^{(2)}$. Then the quantum affine superalgebra $\Uq(\g)$ acts irreducibly
on the vector space  $V$ with the action defined by
\begin{eqnarray}\label{eq:action-1}
\begin{aligned}
&\ga^{1/2}\mapsto q^{1/2},\ \  \qk_i^{ 1/2}\mapsto (\varpi_i q^{\alpha_i})^{1/2}, \ \  \h_{i,s}\mapsto \HH^q_i(s),\\
&  \e_{i,k}\mapsto \widetilde{\X}^{+}_i(k),\ \ \f_{i,k}\mapsto \vartheta_i \widetilde{\X}^{-}_i(k),\\
&\forall i=1, \dots, n, \ \ s\in\Z\backslash\{0\}, \ \ k\in\Z,
\end{aligned}
\end{eqnarray}
where $\varpi_i=\sqrt{-q}$ if $\g=\osp(1|2n)^{(1)}$ and $i=n$,  and $\varpi_i=1$ otherwise; $\vartheta_i=\frac{q_i+q_i^{-1}}{q_i-q_i^{-1}}\varpi_i^{-1}$ if $i=n$, and $\vartheta_i=1$ otherwise.
\end{theo}

\begin{rem}\label{rem:algebra automorphism} The vertex operator representation in Theorem \ref{them:v.o-q} can be changed to that in Theorem \ref{them:v.o} by the following automorphism of $\Uq(\g)$:
\[\begin{aligned}
&\ga\mapsto -\ga,\quad \e_{i,k}\mapsto \e_{i,k},\quad \f_{i,k}\mapsto (-1)^k\f_{i,k},\\
        &\qk_i^{\pm1/2}\mapsto \qk_i^{\pm1/2},\quad \h_{i,k}\mapsto (-1)^{|k|/2}\h_{i,k}, \quad \hh^{\pm}_{i,k}\mapsto (-1)^{\pm k/2}\hh^{\pm}_{i,k}.
\end{aligned}
\]
\end{rem}


\section{Finite dimensional irreducible representations}
In this section, we classify the finite dimensional irreducible representations of the quantum affine superalgebra $\Uq(\g)$ for the affine Lie superalgebras $\g$ in \eqref{eq:g}. We always assume that $\Uq(\g)$-modules are $\Z_2$-graded (cf. Remark \ref{rem:gradings}).

We choose the normalisation for the bilinear form on the weight space of $\g$ so that
$
(\alpha_n,\alpha_n)=1.
$



\subsection{Classification of finite dimensional simple modules}
We fix the triangular decomposition \eqref{eq:triangular-2} for $\Uq(\g)$, and consider highest weight $\Uq(\g)$-modules with respect to this
triangular decomposition.

Let $v_0$ be a highest weight vector in a highest weight  $\Uq(\g)$-module, then for all $i$ and $r$,
\begin{eqnarray}\label{eq:def-hw-vector}
\begin{aligned}
\e_{i,r}\cdot v_0=0,\quad \hh^{\pm}_{i,r}\cdot v_0=\up^{\pm}_{i,r}v_0,\quad \ga^{1/2}\cdot v_0=\up^{1/2}v_0,
\end{aligned}
\end{eqnarray}
for some scalars $\up^{\pm}_{i,r}$ and $\up^{1/2}$, where $\up^{1/2}$ is invertible and $\up^{+}_{i,0}\up^{-}_{i,0}=1$. We define the following formal power series in a variable $x$
\[
\up^{+}_i(x):=\sum_{r=0}^{\infty}\up^{+}_{i,r}x^r, \quad   \up^{-}_i(x):=\sum_{r=0}^{\infty}\up^{-}_{i,-r}x^{-r}, \quad \forall i.
\]

A $\Uq(\g)$-module is said to be at level $0$ if $\gamma$ acts by $\pm\rm{id}$.

By considering the commutation relations of $\hh^{\pm}_{i,r}$,
it is easy to show  \cite{CP0, CP1}
that finite dimensional modules must be at level $0$.
The proof of \cite[Proposition 3.2]{CP0} can be adapted verbatim to prove the following result.
\begin{prop}\label{prop:f.m-hwr}
Every finite dimensional simple $\Uq(\g)$-module  is a level $0$ highest weight module
with respect to the triangular decomposition \eqref{eq:triangular-2}.
\end{prop}

The following theorem is the main result of this section.  Its proof will be given in Section \ref{sect:proof-fd}.

\begin{theo}\label{theo-finite module}
Let $\g=\osp(1|2n)^{(1)}$, $\Sl(1|2n)^{(2)}$ and $\osp(2|2n)^{(2)}$. A simple $\Uq(\g)$-module $V$ is finite dimensional if and only if it can be twisted by some automorphism $\iota_\varepsilon$ into a level $0$ simple highest weight
module satisfying the following conditions:
there exist polynomials $P_i\in\C[x]$ $(i=1, 2, \dots, n)$ with constant term $1$ such that
\begin{eqnarray}\label{eq:hw-polys}
\begin{aligned}
&\up^{+}_i(x) =t_i^{c_i\cdot {\rm deg} P_i}\frac{P_i\left((-1)^{n-i}t_i^{-2c_i}x^{d_i}\right)}{P_i\left((-1)^{n-i}x^{d_i}\right)}=\up^{-}_i(x),
\end{aligned}
\end{eqnarray}
where the equalities should be interpreted as follows: the left side is equal to the middle expression expanded at $0$, and the right side to that expanded at $\infty$.
In the above, $t_i = \left(\sqrt{-1}q^{1/2}\right)^{(\alpha_i, \alpha_i)}$,  and $c_i$ and $d_i$ are defined by
\[\begin{array}{l l}
\g= \osp(1|2n)^{(1)}: & d_i = 1, \quad c_i=\begin{cases} 1, & i\neq n, \\
                                                             2, & i=n;
                                       \end{cases}\\
\g=\Sl(1|2n)^{(2)}: &d_i=c_i=1;\\
\g=\osp(2|2n)^{(2)}:  &d_i=c_i=\begin{cases} 1, & i=n, \\
                                                             2, & i\neq n.
                                       \end{cases}
\end{array}
\]
\end{theo}

As an immediate corollary of Theorem \ref{theo-finite module}, we have the following result.
\begin{coro}\label{prop:f.m-type}
Every finite dimensional simple  $\Uq(\g)$-module can be obtained from  a level $0$
type-{\bf {1}} or type-{\bf {s}} module  by twisting $\Uq(\g)$ with some automorphism given in \eqref{eq:auto-1}.
\end{coro}


\begin{rem} \label{rem:gradings}
Any non $\Z_2$-graded simple highest weight $\Uq(\g)$-module can be regarded as graded by simply assign a parity to its highest weight vector.
\end{rem}


\subsection{Proof of Theorem \ref{theo-finite module}}\label{sect:proof-fd}
The theorem can be proven directly by using the method of \cite{CP1, CP2}. However, there is an easier approach based on quantum correspondences between affine Lie superalgebras developed in \cite{Z92b, Z2, XZ}. The
quantum correspondences allow one to translate results on finite dimensional simple modules of ordinary quantum affine algebras in \cite{CP1, CP2} to the quantum affine superalgebras under consideration. We will follow the latter approach here.

\subsubsection{Facts on ordinary quantum affine algebras}
Corresponding to each $\g$ in \eqref{eq:g},  we have an ordinary (i.e., non-super) affine Lie algebra $\g'$  given in Table \ref{tbl:g}, which has the same Cartan matrix as $\g$.
\vspace{-1mm}
\begin{table}[h]
\renewcommand{\arraystretch}{1.2}
\caption{Table 1. Quantum correspondence}
\label{tbl:g}
\begin{tabular}{c|c|c|c}
\hline
$\g$ & $\osp(1|2n)^{(1)}$    & $\Sl(1|2n)^{(2)}$ & $\osp(2|2n)^{(2)}$  \\
\hline
$\g'$ & $A_{2n}^{(2)}$   & $B_n^{(1)}$  & $D_{n+1}^{(2)}$   \\
\hline
\end{tabular}
\end{table}

\noindent
We denote by $\{\alpha'_1, \dots, \alpha'_n\}$ the set of simple roots realising the Cartan matrix of $\g'$, and normalize the  bilinear form on the weight space of $\g'$ so that  $(\alpha'_n,\alpha'_n)=1$

Let $\x^\pm_{j, r}$, $\hh'^{\pm}_{i,r}$, and $\gamma'^{\pm1/2}$
($1\le i, j\ge n$, \ $r\in\Z$)  be the generators of the quantum affine algebra $\U_t(\g')$ over $\C(t^{1/2})$ (see \cite{CP1, CP2} for details).
Highest weight $\U_t(\g')$-modules are defined in a similar way as  for $\Uq(\g)$ earlier.
A highest weight $\U_t(\g')$ is generated by a highest weight vector $v'_0$, which satisfies
\begin{eqnarray}\label{eq:V'-hwv}
\x^{+}_{i,r}\cdot v'_0=0,\quad \hh'^{\pm}_{i,r}\cdot v'_0=\up'^{\pm}_{i,r}v'_0,\quad \ga'^{1/2}\cdot v'_0={\up'}^{1/2} v'_0,
\end{eqnarray}
where $\up'^{\pm}_{i,r}\in\C$, with ${\up^\prime}^{1/2}\in\C^*$ and ${\up^\prime}^{+}_{i,0}{\up^\prime}^{-}_{i,0}=1$. The module is at level $0$ if $\up'=\pm1$.

Recall that weight modules over $\U_t(\g')$ can always be twisted to type-{\bf 1} modules by automorphisms analogous to \eqref{eq:auto-1}.
The following result is proved in \cite{CP1,CP2}.
\begin{prop}[\cite{CP1,CP2}]\label{prop-fin.dim}
Let $\g'=A_{2n}^{(2)}, B_n^{(1)}, D_{n+1}^{(2)}$.  Every finite dimensional simple $\U_t(\g')$-module is a highest weight module at level $0$.
A level $0$ simple highest weight  $\U_t(\g')$-module of type-{\bf 1} is finite dimensional if and only if there exist polynomials $Q_i\in\C[x]$ ($1\le i\le n$) with constant term $1$ such that
\begin{align}\label{eq:ordinary-hw-polys}
\sum_{r=0}^{\infty}\up'^{+}_{i,r}x^r=t_i^{c_i\cdot {\rm deg} Q_i}\frac{Q_i\left(t_i^{-2c_i}x^{d_i}\right)}{Q_i(x^{d_i})}=\sum_{r=0}^{\infty}\up^{-}_{i,-r}x^{-r},
\end{align}
which holds in the same sense as \eqref{eq:hw-polys}.
Here $t_i=t^{(\alpha'_i,\alpha'_i)/2}$, and the constants $c_i$ and $d_i$ are those defined in Theorem \ref{theo-finite module} for the affine superalgebra $\g$ corresponding to $\g'$ in Table \ref{tbl:g}.
\end{prop}

\subsubsection{Quantum correspondences}
Let $\mathrm{G}$ be the direct product of the groups $\Z_2$ generated by $\sigma_i(1\le i\le n)$ such that ${\sigma_i}^2=1$. We define a $G$-action on $\Uq(\g)$ by
\[\begin{aligned}
&\sigma_i\cdot\ef_{j, r}=(-1)^{(\alpha_i,\alpha_j)}\ef_{j, r},
\quad \sigma_i\cdot\hh^{\pm}_{j, r}= \hh^{\pm}_{j, r},
\quad  \sigma_i\cdot\gamma^{1/2}=\gamma^{1/2},
\end{aligned}\]
for all $i, j\ge 1$ and $r\in\Z$, and form the smash product
$\fU_q(\g):=\Uq(\g)\sharp{K}\mathrm{G}$, which is a Hopf superalgebra.
Similarly,  let
$\mathrm{G'}$ be the direct product of group $\Z_2$ generated by $\sigma'_i(1\le i\le n)$ with ${\sigma'_i}^2=1$, which acts on  $\U_t(\g')$ by
\[ \begin{aligned}
&\sigma'_i\cdot\x^\pm_{j, r}=(-1)^{(\alpha'_i,\alpha'_j)}\x^\pm_{j, r},
\quad \sigma'_i\cdot\hh'^{\pm}_{j, r}= \hh'^{\pm}_{j, r},
\quad  \sigma'_i\cdot{\gamma'}^{1/2}={\gamma'}^{1/2}.
\end{aligned}\]
We also form the smash product $\fU_t(\g'):=\U_t(\g')\sharp{K}\mathrm{G'}$ Hopf algebra.

Now set $t=-q$ (thus $t^{1/2}=\sqrt{-1}q^{1/2}$). The following result is a special case of \cite[Theorem 1.1]{XZ}.
\begin{prop}\cite{XZ, XZ1, Z2}\label{prop:corr}
For each pair $(\g, \g')$ in Table \ref{tbl:g},
there is an isomorphism $\varphi: \fU_q(\g)\longrightarrow \fU_{-q}(\g')$
of associative algebras given by
\begin{equation}\label{eq:dr-map}
\begin{aligned}
& \gamma^{1/2}\mapsto{\gamma'}^{1/2}, \quad
\sigma_i\mapsto \sigma'_i, \quad
{\hh}^\pm_{i,r} \mapsto (-1)^{(n-i) r \epsilon}\sigma'_i \hh'^{\pm}_{i, r},\\
&     \e_{i,r} \mapsto (-1)^{(n-i) r \epsilon} \left(\prod_{k=i+1}^{m+n}\sigma'_k\right)  \x^{+}_{i,r},
\quad \f_{i,r}  \mapsto (-1)^{(n-i) r \epsilon} \left(\prod_{k=i+1}^{m+n}\sigma'_k\right)  \x^{-}_{i,r},
\end{aligned}
\end{equation}
where $\epsilon=1/2$ if $\g=\osp(2|2n)^{(2)}$,  and  1 otherwise.
\end{prop}
The map $\varphi$ becomes a Hopf superalgebra isomorphism up to picture changes and Drinfeld twists; see \cite[Theorem 1.2]{XZ} for details.
Note that the proposition is \cite[Theorem Theorem 3.5]{XZ1}
 stated for $\hh_{i,r}^\pm$ instead of $\kappa_{i,s}$  and with the $o(i)$ there worked out explicitly.

\begin{rem}
The same type of Hopf superalgebra isomorphisms, referred to as quantum correspondences in \cite{XZ}, exist for a much wider range of affine Lie superalgebras
\cite[Theorem 1.2]{XZ}.
Some of them appear as S-dualities in string theory as discovered in \cite{MW}.
\end{rem}

\subsubsection{Proof of Theorem \ref{theo-finite module}}
With the preparations above, we can now prove Theorem \ref{theo-finite module}.
By using Proposition \ref{prop:corr}, we can identify the categories of $\fU_q(\g)$-modules and $\fU_{-q}(\g')$-modules. Then Theorem \ref{theo-finite module} is equivalent to Proposition \ref{prop-fin.dim} under this identification.  Let us describe this in more detail.

If a $\fU_{-q}(\g')$-module is generated by a $\U_{-q}(\g')$ highest weight vector that is an eigenvector of the $\sigma'_i$,
it restricts to a simple $\U_{-q}(\g')$-module.  All highest weight $\U_{-q}(\g')$-modules can be obtained this way, but note that different $\fU_{-q}(\g')$-modules of this type may restrict to the same $\U_{-q}(\g')$-module.
Also any highest weight $\U_{-q}(\g')$-module $V'$ with a highest weight vector $v'_0$ can be lifted to a
$\fU_{-q}(\g')$-module by endowing $\C(q^{1/2}) v'_0$ with a  $G'$-module structure (there are many possibilities).
The same discussion applies to $\fU_q(\g)$- and $\U_q(\g)$-modules.

Assume that $V'$ is a simple $\fU_{-q}(\g')$-module generated by a $\U_{-q}(\g')$ highest weight vector $v'_0$ such that $\C(q^{1/2}) v'_0$ is the $1$-dimensional trivial $G'$-module. Then $V'$ is finite dimensional if and only if
it is at level $0$ and the scalars $\up'^{\pm}_{i,r}$ (cf. \eqref{eq:V'-hwv}) satisfy the given condition of Proposition \ref{prop-fin.dim} for some monic polynomials $Q_i$ with $t^{1/2}=\sqrt{-1}q^{1/2}$.
By Proposition \ref{prop:corr},  $V'$ naturally admits the $\fU_q(\g)$-action
\[
\fU_q(\g)\otimes V' \longrightarrow V', \quad x\otimes V'\mapsto \varphi(x)v', \quad \forall x\in \fU_q(\g), \ v'\in V'.
\]
It restricts to a simple highest weight $\U_q(\g)$-module at level $0$ such that
\[
\hh^{\pm}_{i,r}\cdot v'_0=\up^{\pm}_{i,r}v'_0, \
\text{\ \ with\ \ } \up^{+}_{i,r}=(-1)^{(n-i)r\epsilon}\up'^{+}_{i,r}.
\]
Clearly, the $\up^{+}_{i,r}$ satisfy the condition given in Theorem \ref{theo-finite module}.

As a $\fU_q(\g)$-module, $V'$ is naturally $\Z_2$-graded.
Recall from \cite{XZ} that there exists an element $u\in G$ which is the grading operator in the sense that $u x u^{-1} = (-1)^{[x]} x$ for all homogeneous $x\in\Uq(\g)$. The even and odd subspaces of $V'$ are then the $\pm 1$-eigenspaces of $u$.

The above arguments go through in the opposite direction, i.e., from
$\fU_q(\g)$-modules to $\fU_{-q}(\g')$-modules.  This proves Theorem \ref{theo-finite module}.


\section*{Acknowledgements}
This research was supported by National Natural Science Foundation of China Grants No. 11301130 and No. 11431010,
and Australian Research Council Discovery-Project Grants DP140103239 and DP170104318.


\begin{thebibliography}{9999}


\bibitem{BKK} Benkart, G.; Kang, S.-J.; Kashiwara, M.,
Crystal bases for the quantum superalgebra $\Uq(gl(m,n))$.
{\sl J. Amer. Math. Soc. \bf 13} (2000), no. 2, 295--331.

\bibitem{BGZ} Bracken, A. J.; Gould, M. D.; Zhang, R. B.,
Quantum supergroups and solutions of the Yang-Baxter Equation.
{\sl Modern Physics Letters \bf A5} (1990), 831--840.

\bibitem{CP0} Chari, Vyjayanthi; Pressley, Andrew, Quantum affine algebras. {\sl Comm. Math. Phys.  \bf {142}} (1991), no. 2, 261Ã¢â‚¬â€œ283.

\bibitem{CP1} Chari, Vyjayanthi; Pressley, Andrew, Quantum affine algebras and their representations. {\sl Representations of groups} (Banff, AB, 1994), 59--78, CMS Conf. Proc., 16, Amer. Math. Soc., Providence, RI, 1995.

\bibitem{CP2} Chari, Vyjayanthi; Pressley, Andrew, Twisted quantum affine algebras. {\sl Comm. Math. Phys. \bf {196}} (1998), no. 2, 461--476.

\bibitem{De} Damiani, Ilaria, Drinfeld realisation of affine quantum algebras: the relations. {\sl Publ. Res. Inst. Math. Sci. \bf{ 48}} (2012), no. 3, 661--733.



\bibitem{Dr}  Drinfeld, V. G.,  A new realisation of Yangians and of quantum affine algebras. (Russian) {\sl Dokl. Akad. Nauk SSSR \bf {296} } (1987), no. 1, 13--17; translation in {\sl Soviet Math. Dokl. \bf{36}} (1988), no. 2, 212--216


\bibitem{Jn1}  Jing, Naihuan, Twisted vertex representations of quantum affine algebras. {\sl Invent. Math.\bf{102}} (1990), no. 3, 663--690.


\bibitem{JnM} Jing, Naihuan; Misra, Kailash C., Vertex operators for twisted quantum affine algebras. {\sl Trans. Amer. Math. Soc. \bf{ 351}} (1999), no. 4, 1663--1690.

\bibitem{Jn96}Jing, Naihuan, Higher level representations of the quantum affine algebra $\Uq (\hat{sl} (2))$. {\sl Journal of Algebra \bf 182} (1996),448-468.

\bibitem{K78}Kac, Victor G.,  Infinite-dimensional algebras, Dedekind’s η-function, classical Möbius function and the very strange formula.
{\sl Adv. Math. \bf 30} (1978) 85--136.

\bibitem{K98}Kac, Victor G., Vertex algebras for beginners. Second edition. University Lecture Series, {\bf 10}.  American Mathematical Society, Providence, RI, 1998.

\bibitem{LGZ}
Links, J. R.; Gould, M. D.; Zhang, R. B. Quantum supergroups, link polynomials and representation of the braid generator. {\sl Rev. Math. Phys. \bf 5} (1993), no. 2, 345--361.

\bibitem{LP} Lepowsky, James; Primc, Mirko, Standard modules for type one affine Lie algebras. Number theory (New York, 1982), 194–251, Lecture Notes in Math., 1052, Springer, Berlin, 1984.

\bibitem{MW} Mikhaylov, V; Witten, E.; Branes and Supergroups.
{\sl Commun. Math.  Physics, \bf 340} (2015) no. 2, 699--832.


\bibitem{PSV} T. D. Palev, N. I. Stoilova, J. Van der Jeugt, Finite-dimensional representations
of the quantum superalgebra $\U_q(\mathfrak{gl}(n|m))$ and related $q$-identities.
{\sl  Commun. Math. Phys. \bf 166}   (1994),  no. 2, 367--378.

\bibitem{WZ} Wu, Yuezhu; Zhang, R. B., Integrable representations of the quantum affine special linear superalgebra. {\sl
  Adv. Theor. Math. Phys. \bf 20} (2016), no. 3, 553--593.

\bibitem{XZ} Xu, Ying;  Zhang, R. B.,
Quantum correspondences of affine Lie superalgebras.  {\sl Math. Research Lett.}, in press; arXiv:1607.01142.

\bibitem{XZ1}	Xu, Ying;  Zhang, R. B., Drinfeld realisations of quantum affine superalgebras.  arXiv:1611.06449

\bibitem{Y94} Yamane, H.,   Quantized enveloping algebras associated with simple
Lie superalgebras and their universal $R$- matrices.
 {\sl Publ. Res. Inst. Math. Sci. \bf{30}} (1994) 15-87.

\bibitem{Y99} Yamane, H.,  On definding relations of the affine Lie superalgebras
and their quantized universal enveloping superalgebras.
 {\sl Publ. Res. Inst. Math. Sci. \bf{35}} (1999) 321-390.

\bibitem{Zhc} Zhang, Hechun, The quantum general linear supergroup, canonical bases and Kazhdan-Lusztig polynomials. {\sl Sci. China Ser. \bf A 52} (2009), no. 3, 401--416.

\bibitem{Zh} Zhang, Huafeng,
Representations of quantum affine superalgebras. {\sl Math. Z. \bf 278}
 (2014), 663--703.

\bibitem{Z92a}  Zhang, R. B.,   Braid group representations arising from quantum supergroups
with arbitrary q and link polynomials.
{\sl J. Math. Phys. \bf 33} (1992), no. 11, 3918--3930.

\bibitem{Z92b}   Zhang, R. B.,
Finite-dimensional representations of $\U_q(osp(1/2n))$
and its connection with quantum $so(2n+1)$.
{\sl Lett. Math. Phys. \bf 25} (1992), no. 4, 317--325.

\bibitem{Z92c} Zhang, R. B., Universal $L$ operator and invariants of the quantum supergroup $\Uq(gl(m/n))$. {\sl J. Math. Phys. \bf 33} (1992), no. 6, 1970--1979.

\bibitem{Z93} Zhang, R. B.,  Finite-dimensional irreducible representations of
the quantum supergroup  $\U_q(gl(m/n))$.
{\sl  J. Math. Phys. \bf 34} (1993), no. 3, 1236--1254.

\bibitem{Z93b} Zhang, R. B., Finite-dimensional representations of $\Uq(C(n+1))$ at arbitrary $q$. {\sl J. Phys. \bf A 26} (1993), no. 23, 7041--7059.

\bibitem{Z95}  Zhang, R. B.,
Quantum supergroups and topological invariants of three-manifolds.
{\sl Rev. Math. Phys. \bf 7} (1995), no. 5, 809--831.

\bibitem{Z96} Zhang, R. B.,   The $gl(M|N)$ super Yangian and its finite-dimensional representations. {\sl Lett. Math. Phys. \bf 37} (1996), no. 4, 419--434.

\bibitem{Z2} Zhang, R. B., Symmetrizable quantum affine superalgebras and their representations.  {\sl J. Math Phys. \bf{38}} (1997),  535--543.

\bibitem{Z98} Zhang, R. B.,
Structure and representations of the quantum general linear supergroup.
{\sl Commun. Math. Phys. \bf 195}  (1998)  525 -- 547.


\bibitem{ZBG91} Zhang, R. B.; Bracken, A. J.;  Gould, M. D.,
Solution of the graded Yang-Baxter equation associated with
the vector representation of $\U_q(osp(M/2n))$.
{\sl Phys. Lett. \bf B 257} (1991), no. 1-2, 133-139.

\bibitem{ZGB91b}
Zhang, R. B.; Gould, M. D.; Bracken, A. J.,
Solutions of the graded classical Yang-Baxter equation and integrable models.
    {\sl J. Phys. \bf  A 24} (1991), 1185--1197.

\bibitem{Zo}  Y. M. Zou, Integrable representations of $\U_q(osp(1,2n))$.
    {\sl J. Pure Appl. Algebra \bf 130} (1998), no. 1, 99--112.
\end{thebibliography}
\end{document}